\documentclass[reqno,12pt,a4paper]{amsart}
\usepackage{pgfplots}
\usepackage[english]{babel}
 \pdfoutput=1
\usepackage{amsmath,amsthm,amssymb,amsfonts, mathdots}
\usepackage{scrtime, cancel}
\usepackage{enumitem, color}

\definecolor{pinegreen}{rgb}{0.0, 0.47, 0.44}
\definecolor{brilliantrose}{rgb}{1.0, 0.33, 0.64}
\usepackage{mathtools}
\usepackage{ifpdf}
\ifpdf
\usepackage[backref]{hyperref}
\else
\usepackage[hypertex]{hyperref}
\fi
\usepackage{pdfsync,verbatim}
\usepackage{color}
\mathtoolsset{showonlyrefs}

\numberwithin{equation}{section}   

\allowdisplaybreaks
\newtheorem{theorem}{Theorem}[section]
\newtheorem{lemma}[theorem]{Lemma}
\newtheorem{proposition}[theorem]{Proposition}

\theoremstyle{definition}
\newtheorem{remark}[theorem]{Remark}

\usepackage{graphicx}

\makeatletter
\newcommand*\bigcdot{\mathpalette\bigcdot@{.5}}
\newcommand*\bigcdot@[2]{\mathbin{\vcenter{\hbox{\scalebox{#2}{$\m@th#1\bullet$}}}}}
\makeatother

\newcommand{\R}{\mathbb{R}}

\newcommand{\N}{\mathbb{N}}
\newcommand{\Z}{\mathbb{Z}}
\newcommand{\ellipses}{E}
\def\misgausskd{d \gamma_\infty}
\def\misgaussk{\gamma_\infty}
\newcount\dotcnt\newdimen\deltay
\def\Ddot#1#2(#3,#4,#5,#6){\deltay=#6\setbox1=\hbox to0pt{\smash{\dotcnt=1
\kern#3\loop\raise\dotcnt\deltay\hbox to0pt{\hss#2}\kern#5\ifnum\dotcnt<#1
\advance\dotcnt 1\repeat}\hss}\setbox2=\vtop{\box1}\ht2=#4\box2}
\def\succvarepsilon{{\underline{\varepsilon}}}

\def\Blue{\color{blue}}

\pgfplotsset{compat = newest}

\makeatletter
\def\author@andify{%
  \nxandlist {\unskip ,\penalty-1 \space\ignorespaces}%
    {\unskip {} \@@and~}%
    {\unskip \penalty-2 \space \@@and~}%
}
\makeatother

\title[Variation for Ornstein--Uhlenbeck: the higher--dimensional case]{
Variational inequalities \\ for the Ornstein--Uhlenbeck semigroup: \\the higher--dimensional case}

\author{Valentina Casarino}
\address{DTG, Universit\`a degli Studi di Padova\\ Stradella san Nicola 3 \\I-36100 Vicenza \\ Italy}
\email{valentina.casarino@unipd.it}
\author{Paolo Ciatti}
\address{Dipartimento di Matematica "Tullio Levi Civita", Universit\`a degli Studi di Padova\\Via Trieste, 63, 35131 Padova,  \\ Italy}
\email{paolo.ciatti@unipd.it}
\author{Peter Sj\"ogren}
\address{Mathematical Sciences,  University of Gothenburg and  Mathematical Sciences,
Chalmers University of Technology  \\ SE - 412 96 G\"oteborg, Sweden}
\email{peters@chalmers.se}



\thanks{The first and second authors are members of the Gruppo Nazionale per l'Analisi Matematica, la Probabilità e le loro Applicazioni (GNAMPA)
of the Istituto Nazionale di Alta Matematica (INdAM) and were partially supported by Progetto GNAMPA 2024: ``$L^p$ estimates for
singular integrals in nondoubling settings". The third author would like to thank the University of Padova  for its generosity during his visits.}

\date{\today}

\begin{document}

\maketitle

\begin{abstract}
We study the $\varrho$-th order variation seminorm of a general Ornstein--Uhlenbeck semigroup $\left(\mathcal H_t\right)_{t>0}$ in $\R^n$,  taken with respect to $t$.
We prove that this seminorm defines an operator of weak type $(1,1)$ with respect to the invariant  measure when $\varrho> 2$. For large $t$,
one has 
 an enhanced  version  of the standard weak-type $(1,1)$ bound.
  For small $t$,  the proof hinges on vector-valued Calder\'on--Zygmund techniques in the local region, and
 on  the fact that
  the $t$ derivative of the integral kernel of  $\mathcal H_t$ in the global region has a bounded number of zeros in $(0,1]$.
 A counterexample is given  for $\varrho= 2$;
  in fact, we
 prove that the second  order variation seminorm of  $\left(\mathcal H_t\right)_{t>0}$, 
 and therefore also the $\varrho$-th order variation seminorm for any $\varrho\in [1,2)$,
 is not of strong nor weak type $(p,p)$ for any $p \in [1,\infty)$ with respect to the invariant measure.

  \end{abstract}

\section{Introduction}
In this paper we prove
 the   weak type $(1,1)$  for the variation  operator of a general Ornstein--Uhlenbeck semigroup $\big( \mathcal H_t\big)_{t> 0}$  in $\R^n$ for any $n\ge 1$.

Recently the authors proved
 this result       in dimension one   \cite{CCS6}, answering a question asked by Almeida et al.\ in \cite[p.\,31]{Almeida}.
In this article we provide an answer to the question in  \cite{Almeida}
covering the higher--dimensional case.
The methods we use are different from those of \cite{CCS6}, which seem hard to adapt to the case $n>1$.

 The proof relies on extensive application of properties of the Ornstein--Uhlenbeck semigroup and its integral kernel $K_t(x,u)$ (known as the Mehler kernel),
recently proved by the authors in a series of papers concerning various  issues arising from harmonic analysis 
(see \cite{CCS1, CCS2, CCS3, CCS4, CCS5, CCS6} and also the brief introductory summary \cite{CCS7}). We also apply vector--valued Calder\'on--Zygmund theory.
As far as we know, the first application of this theory in this context was in \cite{Harboure2, Harboure} (see also \cite{Rubio}, 
where  \cite{Harboure2, Harboure}   have their roots, and \cite{Crescimbeni}).

Let $\big(T_t\big)_{t> 0}$
be a family of bounded operators between spaces of functions defined on a measure space $(X,\mu)$. 
The $\varrho$-th order variation operator of $\big(T_t\big)_{t> 0}$  on an interval $I\subset \R_+$ is defined,
 for $1 \le \varrho < \infty$ and suitable $f$,  by
           \begin{equation}\label{def:intro_rho2}
  \|T_t\,f(x)\|_{v(\varrho), I} = \sup\, \left( \sum_{i=1}^N |T_{t_i} f(x)- T_{t_{i-1}}f(x)|^\varrho \right)^{1/\varrho},\qquad x\in X,
\end{equation}
where the supremum is taken over all finite, increasing
 sequences $\left(t_i \right)_0^N$ of points in $I$.

In the last fifty years,
variational inequalities, stating that  the $L^p (\mu)$ norm of \eqref{def:intro_rho2}
 is  bounded by a constant times the $L^p(\mu)$ norm of $f$,
 have been 
 widely investigated. The first bounds, due to  D. Lépingle \cite{Lepingle} for  bounded martingales and to
 V. F.~Gaposhkin \cite{Gaposhkin1, Gaposhkin2} and J. Bourgain  \cite{Bourgain} for ergodic averages,
 have been generalized in various directions.
 We refer to \cite{CCS6} for a brief description of  recent results 
 with a focus on harmonic analysis; see \cite{Seeger, Campbell, Jones, Jones1, Jones1-higher,  Jones2, Jones4, Harboure, Betancor1, Betancor2, Betancor3, {AlmeidaMedit}, Ma1, Ma2}. 
 More recently,
  the study of oscillatory and jump inequalities,  
 in addition
to variational inequalities,
began to develop from both an analytic and a number--theoretic perspective; we refer in particular to
 \cite{Bourgain_etal, Mirek1, Mirek2, Mirek3, Mirek4, Mirek5, Mirek6}.
For an overview of  the connections
with ergodic theory, analytic number theory, and harmonic analysis, see in particular  \cite{Krause}.

In this paper
the  family $\big(T_t\big)_{t> 0}$  in \eqref{def:intro_rho2} will  always be  the Ornstein-Uhlenbeck semigroup $\big( \mathcal H_t\big)_{t> 0}$ in $\R^n$.
This is the semigroup  generated by the elliptic operator
\begin{equation}\label{def:opL}
\mathcal L=
\frac12
\mathrm{tr}
\big( Q\nabla^2 \big)+\langle Bx, \nabla \rangle,
\end{equation}
called the Ornstein-Uhlenbeck operator.
Here $\nabla$ is the gradient and $\nabla^2$ the Hessian matrix. Moreover,
 $Q$ is  a real, symmetric and positive definite $n\times n$
 matrix, called the covariance of $\mathcal L$, and
$B $ is a  real $n\times n$   matrix whose  eigenvalues have negative real parts;
$B $  indicates the drift of  $\mathcal L$.
In Section \ref{preliminaries}
we will provide  explicit expressions for $\mathcal H_t$, seen as an integral operator with  a kernel $K_t(x,u)$.
It is well known  that there exists an  invariant measure under the action of $\mathcal H_t$, unique up to a positive factor. This measure, denoted  by  $\gamma_\infty$, is basic in  Ornstein-Uhlenbeck theory; it is described explicitly in Section \ref{preliminaries}.

In 2001 Jones and Reinhold \cite{Jones1}  proved that for    $\varrho > 2$ the variation operator
of any symmetric
 diffusion semigroup is  $L^p$ bounded for   $1<p<\infty$. Ten years later Le Merdy and Xu   \cite{Le Merdy}
 extended this result to a nonsymmetric context. In fact, Corollary 4.5 in   \cite{Le Merdy}
 applies to
 $\mathcal H_t$ (see \cite[p. 31]{Almeida} for a discussion).  It says that the operator given by
 $f \mapsto  \| \mathcal H_t f(x)\|_{v(\varrho),\Bbb R_+}$,
 is bounded on $L^p (\gamma_\infty)$ for $1<p<\infty$ and   $\varrho > 2$.
 Our result gives the corresponding weak type  $(1,1)$, as follows.

\begin{theorem}\label{thm}
 For each  $\varrho > 2$ the operator that maps   $f \in L^1(\gamma_\infty)$ to the function
  \begin{equation*}                    
  \| \mathcal H_t f(x)\|_{v(\varrho),\Bbb R_+}, \quad x \in \mathbb R^n,
\end{equation*}
where the $v(\varrho)$ seminorm is taken in the variable $t$, is of weak type $(1,1)$ with respect to the measure $\gamma_\infty$.
\end{theorem}
In other words, the inequality
\begin{equation} \label{main-ineq}
\gamma_\infty
\{x\in\R^n :   \| \mathcal H_t f(x)\|_{v(\varrho),\Bbb R_+}
     > \alpha\} \le \frac{C}\alpha\,\|f\|_{L^1( \gamma_\infty)}, \qquad \alpha>0,
\end{equation}
 holds for some $C > 0$ and all functions $f\in L^1 (\gamma_\infty)$.

The  pointwise estimates of the Mehler kernel and  its derivative needed  to  prove  Theorem \ref{thm}
 are quite different  for small and large $t$. Thus we are  led to
 distinguish between the variation in the interval $0 < t \le 1$ and that in  $1 \le t < \infty$.
 
 The variation in $[1,\infty)$ is treated following a geometric approach
 recently developed by the authors  in \cite{CCS1, CCS2, CCS3, CCS5}, which relies on a  system of polar-like coordinates. Proposition \ref{t>1} is actually an enhanced version of \eqref{main-ineq}  with $\Bbb R_+$ replaced by
 $[1,\infty)$.

 The study of the variation  in $(0,1]$  is more delicate and
requires a further distinction between local and global parts of the Ornstein-Uhlenbeck semigroup operator.
In the Ornstein--Uhlenbeck setting, 
the local part is usually defined by the relation $|x-u| \lesssim 1/(1+|x|)$ between the two variables of the Mehler kernel.
This splitting was first introduced by
 Muckenhoupt  \cite{Mu} in one dimension, and by the third author \cite{Sj} in higher dimension. It has since been widely used in the literature, in particular in \cite{Perez}. In this paper, P\'erez  shows that the local parts of many   operators related to the Ornstein--Uhlenbeck semigroup behave precisely like the corresponding classical operators in Euclidean space.

 The reason for this definition of the local part is that it makes the density of the measure   $\gamma_\infty$  have the same order of magnitude at the two points $x$ and $u$. But this remains true if $x$ and $u$ belong to a suitable elliptic ring of width roughly $1/(1+|x|)$.  In this paper, we will  define the local part by splitting $\R^n$
 into  rings of this type.   This makes the arguments  more explicit; for details see Section \ref{local1}.  The expression ``local part" is not quite adequate here, since
 $x$ and $u$ may be far apart, but we prefer to keep it, since this part plays the same role in the proofs as it does in earlier work.

To prove  \eqref{main-ineq} for  the global part and  $0 < t \le 1$, we  estimate its kernel using a method from
 \cite{CCS5}. It is
based on the observation that   the
 number of zeros in $(0,1]$ of the function $t \mapsto  \partial_t K_t(x,u)$
 is bounded, uniformly in  $(x,u)\in \mathbb R^n\times\mathbb R^n$.

The argument for
 the local part and  $0 < t \le 1$ is based on vector--valued Calder\'on--Zygmund techniques for singular integrals
from \cite{Harboure2, Harboure, Crescimbeni}. This requires a transition to Lebesgue measure.

The parameter $\varrho$ deserves more attention.
Some of our results hold   for $\varrho\ge 1$ as well. This is the case of the  variational bounds  in $[1,\infty)$
and also in  $(0,1]$ for
the global part;  we refer to Theorem \ref{t>1} and Theorem \ref{t<1,global}, respectively.
What forces the restriction $\varrho>2$ in Theorem
\ref{thm} is the variational bound
for
the local part of $\mathcal H_t$  when $t\in (0,1]$;
see Theorem \ref{locsmallt_v2}, where we apply  results from \cite{Le Merdy}  holding only for $\varrho> 2$.
In the last  section, we provide a counterexample showing that  the condition $\varrho>2$ is necessary in Theorem \ref{thm}.

\medskip

 \subsection{Structure of the paper}
 We  gather some known facts about the Ornstein--Uhlenbeck semigroup and its integral kernel $K_t$ in Section~\ref{preliminaries};  in particular,
 we give pointwise bounds  for $K_t$ and its time derivative {$\dot K_t = \partial_t K_t$}.
 In Section~\ref{variation} some  basic properties of the variation seminorm are discussed; here
 we also introduce a few reductions which will simplify the proof of Theorem~\ref{thm}.
 Section \ref{t large}
 is devoted to the proof of the weak type inequality \eqref{main-ineq} with $\R_+$ replaced by $[1,+\infty)$.
The proof of  \eqref{main-ineq}
{\ for $(0,1]$} is carried out in Sections~\ref{local1}, \ref{The variation of the global part} and \ref{tsmall}. More precisely, in Section \ref{local1} we describe our localization procedure and the local and global parts, and  Section~\ref{The variation of the global part} deals with  the global part. The local part is treated in Section
 \ref{tsmall}.  The discussione of a technical issue, arising in the local context and involving  Bochner integral, is postponed to the Appendix (Section \ref{Appendix}). 
  In Section \ref{counterexample} we conclude by showing
  that
 Theorem \ref{thm} does not hold for $\varrho=2$.

     \subsection{Notation}            We will  use the symbols $0<c,\,C<\infty$ to denote constants, not necessarily equal at different occurrences. These constants will depend
                  only on $n$, $Q$  and $B$, unless otherwise explicitly stated.
                  If $a$ and $b$ are positive quantities, $a \lesssim b$ or
equivalently  $b \gtrsim a$ means  $a \le C b$. If both $a \lesssim b$ and  $b \lesssim a$ hold,
we shall write  $a \simeq b$.
By $\N$ we denote the set of all nonnegative integers.

We write   \[\dot K_t = \partial_t K_t,\] that is, we adopt the  dot notation for differentiation with respect to the  variable $t$.

 \section{The Mehler kernel}\label{preliminaries}

{{The Ornstein--Uhlenbeck semigroup may be formally written as
$
\mathcal H_t=e^{t\mathcal L}$,
${t> 0}$, with $\mathcal L$ given by \eqref{def:opL}. In order to give more explicit expressions for this semigroup,
we introduce the
positive definite, symmetric matrices
\begin{equation}                                            \label{defQt}
Q_t=\int_0^t e^{sB}Qe^{sB^*}ds, \qquad \text{ $0<t\leq +\infty$},
\end{equation}
and the
normalized Gaussian measures  $\gamma_t $ in $\R^n$, with $t\in (0,+\infty]$,  whose densities
with respect to Lebesgue measure $dx$ are given by
\[
x\mapsto (2\pi)^{-\frac{n}{2}}
(\text{det} \, Q_t)^{-\frac{1}{2} }
\exp\left({-\frac12 \langle Q_t^{-1}x,x\rangle}\right).
\]
 Here the case $t=+\infty$ gives the invariant measure  $\gamma_\infty$ that appears
 already in the introduction.


Kolmogorov's formula states that
for all bounded and continuous functions in $\R^n$ one has
\begin{equation}\label{Kolmo}
\mathcal H_t
f(x)=
\int
f(e^{tB}x-y)\,d\gamma_t (y)\,, \qquad x\in\R^n,\;\; t>0.
\end{equation}
Starting from \eqref{Kolmo}, one may write $\mathcal H_t $ as an integral operator
with a (Mehler) kernel $K_t(x,u)$.
  We point out that the term kernel  in this paper  refers to integration with respect to our basic measure $ \gamma_\infty$, with an exception in Section  \ref{tsmall}.
In fact, we saw in  \cite{CCS2} that  for each $f\in L^1(\gamma_\infty)$ and all $t>0$ one has
\begin{align}                       \label{def-int-ker}
 \mathcal H_t
f(x) &=
 \int
K_t
(x,u)\,
f(u)\,
 d\gamma_\infty(u)
  \,, \qquad x\in\R^n,
\end{align}
 where for $x,u\in\R^n$ and $t>0$ the Mehler kernel $K_t$   is given by
\begin{align}    \label{mehler}                   
K_t (x,u)\!
=\!
\Big(
\frac{\det \, Q_\infty}{\det \, Q_t}
\Big)^{{1}/{2} }
e^{R(x)}\,
\exp \Big[
{-\frac12
\left\langle \left(
Q_t^{-1}-Q_\infty^{-1}\right) (u-D_t \,x) \,, u-D_t\, x\right\rangle}\Big]\!.\;\;\;\;
\end{align}
Here
 $R(x)$ is
the quadratic form
\begin{equation*}
R(x) ={\frac12 \left\langle Q_\infty^{-1}x ,x  \right\rangle}, \qquad\text{$x\in\R^n$},
\end{equation*}
and
\begin{equation}\label{def:Dt}
 D_t =
 Q_\infty \, e^{-tB^*} Q_\infty^{-1} ,
\end{equation}
 which  is a one-parameter group, defined for all $t \in\R$.

It will be convenient to introduce another norm on $\R^n$ by
\[
|x|_Q := | Q_\infty^{-1/2}\,x|,  \qquad x \in \R^n.
\]
Then
  $|x|_Q \simeq |x|$,
and  $R(x) = |x|_Q^2/2 , \;\,\text{$x\in\R^n$}$.
Observe that the density of     $\gamma_\infty$ is proportional to   $ e^{-R(x)}$.

\subsection{Pointwise estimates for the Mehler kernel}

We shall repeatedly use the following pointwise estimate of the Mehler kernel. For $0<t\leq 1$ and
 all $(x,u) \in \mathbb R^n\times \mathbb R^n$ one has
\begin{equation}\label{litet}
   K_t(x,u)
\,\lesssim \, \frac{ e^{R( x)}}{t^{n/2}} \exp\left(-c\,\frac{|u-D_t\, x |^2}t\right).
\end{equation}
This is proved in \cite[(3.4)]{CCS2}.

We shall  need estimates for the time derivative of the Mehler kernel as well.
For an explicit expression of $\dot K_t$, the reader is referred to \cite[Lemma 4.2]{CCS5}; we only recall here the following pointwise  bounds for $\dot K_t$
(see  \cite[(5.5) and (5.4)]{CCS5}):
\begin{equation}\label{dotKeps}
|\dot K_t(x,u)|\lesssim
\, e^{R(x)}\,t^{-n/2}\,
\exp{\left(-c\,\frac{|u-  D_{t}\,x|^2}t \right)}\,
\left(t^{-1}+|x| t^{-1/2}\right),
 \end{equation}
for $0<t \le 1$, and
\begin{equation}\label{dotK1} |\dot K_t (x,u)| \lesssim \, e^{R(x)}\,\exp{\left(-c\,|  D_{-t}\,u- x|^2\right)}\,\left(|D_{-t}\,u| + e^{-ct}\right) \end{equation}for   $t\geq1$.
Both formulae hold
for all $(x,u) \in \mathbb R^n\times \mathbb R^n$.

\section{The variation seminorm}\label{variation}
To state some basic properties of the   variation seminorm, we define it for a generic continuous
 real-valued function $\phi$  defined in an interval $I$, by setting for any  $1 \le \varrho < \infty$
\begin{equation}                            \label{def_var_gen}
  \|\phi\|_{v(\varrho), I} = \sup \left( \sum_{i=1}^{N} |\phi(t_i) - \phi(t_{i-1})|^\varrho \right)^{1/\varrho}.
\end{equation}
As in \eqref{def:intro_rho2},  the supremum is taken over all finite, increasing sequences $\left(t_i \right)_0^N$ of points in $I$.
         This  seminorm  vanishes only for constant functions.

Observe  that the seminorm
$\|.\|_{v(\varrho), I}$ is  decreasing in  $\varrho$ for $1 \le \varrho < \infty$.
It is  also subadditive in $I$, in the following sense. Take an inner point $\tau$ of $I$ and set $I_+  = I \cap [\tau, +\infty)$ and $I_-  = I \cap (-\infty, \tau]$. Then for  $1 \le \varrho < \infty$ and any $\phi$
\begin{equation*}
   \|\phi \|_{v(\varrho), I} \le \|\phi \|_{v(\varrho), I_+} + \|\phi  \|_{v(\varrho), I_-}.
 \end{equation*}

  If  $\phi \in C^1(I)$ and $\phi' \in L^1(I)$, then  for    $1 \le \varrho < \infty$
 \begin{equation}\label{var}
   \|\phi  \|_{v(\varrho), I} \le \int_{I} |\phi'(t)|\,dt,
 \end{equation}
see \cite[Lemma 2.1]{CCS6}.

In the last section, we will  consider a discrete version of the variation, obtained by replacing $I$ in the   definition                       \eqref{def_var_gen}
by a set $\mathcal I$ which is the intersection of $\N_+$ and an interval, and where $\phi$ is defined.
   For $\varrho = 2$ one has the simple estimate
\begin{equation}\label{discrete_trivial}
  \|\phi\|_{v(2), \mathcal I} \lesssim \left( \sum_{\ell\in \mathcal I} \phi(\ell)^2 \right)^{1/2}.
\end{equation}

For future convenience, we note that
a combination of
 \eqref{var}
 and \cite[Lemma 5.3]{CCS5} yields that for any $\varrho \in [1,\infty)$, any interval $I \subset \R_+$ and all $x\in\R^n$
 \begin{align}  \label{important}
 \|\mathcal H_tf(x)\|_{v(\varrho), I}&\le       
    \int_I \left|\frac{\partial}{\partial t} \int K_t(x,u) f(u)\,d\gamma_\infty (u)  \right|\,dt\notag
 =\int_I \left| \int \dot K_t(x,u) f(u) \,d\gamma_\infty (u) \right|\,dt     \notag  \\
&\le \int         \int_I \big|  \dot K_t(x,u)\big| \,dt \, |f(u)|\,  d\gamma_\infty (u),\qquad f\in L^1(\gamma_\infty).
  \end{align}

     \subsection{Simplifications}\label{Simplificationsss}
By means of a few
observations, it is possible to simplify the proof of Theorem~\ref{thm}.

First of all, when proving the inequality \eqref{main-ineq} one may take
$f $ such that
 $\|f\|_{L^1( \misgaussk)}=1$. As a consequence,  $\alpha$ in the same estimate may be assumed large, for instance $\alpha >2$, since $\gamma_\infty$ is finite.

\smallskip

Moreover, as already observed in the study  of the weak-type $(1,1)$ of operators related to the Ornstein--Uhlenbeck semigroup
 (see \cite[Section 5]{CCS2}),
when proving \eqref{main-ineq}
we mostly need to take into account only points  $x$ belonging to
the ellipsoidal annulus
 \begin{equation*}                         
{\mathcal C_\alpha}=\left\{
x \in\R^n:\, \frac12  \log \alpha\le
R(x)
\le 2  \log \alpha
\right\}.
\end{equation*}

Indeed, the unbounded component of the complement of $\mathcal C_\alpha$ can always be neglected, since
$\gamma_\infty\left( \left\{
x \in\R^n:\,
R(x)
> 2  \log \alpha
\right\}\right)\lesssim 1/\alpha$.

For the variation of $\mathcal H_t$ in $[1,+\infty)$,
the bounded component  of the complement of $\mathcal C_\alpha$ 
is negligible  as well. Indeed, we see  from \cite[(5.4) and Lemma 5.1]{CCS5}
      that  \begin{equation}\label{ineq100}
\int_1^\infty
|\dot K_t (x,u)|\,
dt
\lesssim e^{R(x)}.
\end{equation}
Combined  with   \eqref{important}, where $I=[1,\infty)$  and  $f$  is normalized in $L^1( \gamma_\infty)$,
this leads to
 \begin{align*}                                 
 \|\mathcal H_tf(x)\|_{v(\varrho), [1,+\infty)}&\lesssim e^{R(x)}  \lesssim  \sqrt{\alpha}
  \end{align*}
 for $R(x) < \frac12 \log \alpha$.
 Taking $\alpha$ suitable large, we will have
 $\|\mathcal H_tf(x)\|_{v(\varrho), [1,+\infty)}
<  \alpha$ in the bounded component of the complement of  $\mathcal C_\alpha$.

 \medskip

\section{The variation    of  $\mathcal H_tf(x)$  in   $[1,+\infty)$}\label{t large}

We first introduce the adapted polar coordinates from \cite{CCS2}.
Fix $\beta>0$ and consider the ellipsoid
\begin{equation*}
\ellipses_\beta
=\{x\in\R^n:\, R(x)=
\beta\}
\,.
\end{equation*}
 As a consequence of \cite[formula (4.3)]{CCS2},
the map
$s\mapsto R(D_s z)$ is strictly increasing for each $0 \ne z\in\R^n$.
Thus we may write  any $x\in\R^n,\, x\neq 0$,  uniquely as
\begin{equation*}                   
x=D_s \tilde x
\,,
\end{equation*}
for some $\tilde x\in \ellipses_\beta$
and $s\in\R$, and  $s$ and $\tilde x$ are
 our polar coordinates of $x$.

We  recall  two  results, previously proved by the authors, which will be essential in our arguments.
The following lemma is  \cite[Lemma 5.1]{CCS3} with  $\sigma = 1$;   the factor $1/4$ occurring in \cite{CCS3} is replaced by a generic $\delta > 0$, which causes no problem.
  \begin{lemma}\label{preliminary-t-large-not-kernel}
   \cite{CCS3}
Let $\delta > 0$. For  $x,u\in\R^n$, one has
\begin{equation*}
 \int_1^{+\infty}\exp \Big({-\delta\left|  D_{-t}\,u- x
\right|^2 }\Big)\big|   D_{-t} \,  u\big|\, dt \lesssim
1,                                                                  
\end{equation*}
where the implicit constant  may depend on $\delta$,  $n$, $Q$ and $B$.
\end{lemma}
  \begin{lemma}\label{bounding-GMAa}  \cite[Lemma 7.2]{CCS3}, \cite[Proposition  6.1]{CCS2}
Let $\delta > 0$ and  $\alpha > 2$, and assume $f$ normalized in $L^1 (\gamma_\infty)$. Then
 \[
\gamma_\infty \Big\{ x = D_s\,\tilde x \in\mathcal C_\alpha:\:
e^{R(x)}
\int
\exp
 \big(
- \delta\, \big|
\tilde x-\tilde u\big|^2\big)
\,|f(u)|\,\misgausskd
 (u)
>
\alpha
\Big\}
\lesssim \frac{1}{\alpha\,\sqrt{\log \alpha}}.
\]
                     \end{lemma}
Here the polar coordinates are defined with $\beta = \log \alpha$, the implicit constant is as in Lemma \ref{preliminary-t-large-not-kernel}, and $\sigma = 1$ in \cite[Lemma~7.2]{CCS3}.

We can now deduce the  bound for the variation operator in   $[1,+\infty)$.
\begin{theorem} \label{t>1}
 Let  $1\le \varrho <\infty$. The operator that maps   $f \in L^1(\gamma_\infty)$ to the function
  \begin{equation*}
  \| \mathcal H_t f(x)\|_{v(\varrho), [1,+\infty)}, \quad x \in \mathbb R^n,
\end{equation*}
is of weak type $(1,1)$ with respect to the measure $\gamma_\infty$.
In fact,  the following stronger result holds: if  $\|f\|_{L^1( \gamma_\infty)} = 1$, then
\begin{equation} \label{stronger}
\gamma_\infty
\left\{x\in\R^n:   \| H_t f(x)\|_{v(\varrho),[1,\infty)}
        > \alpha \right\} \lesssim \frac{1}{\alpha \sqrt{\log \alpha}}, \qquad \alpha > 2.
\end{equation}
\end{theorem}

\begin{proof}
This proof is similar to that of Proposition 6.1 in \cite{CCS5}; we sketch it for the sake of completeness.
Let $f$ be normalized in $L^1(\gamma_\infty)$.
In the light of the considerations made in Subsection \ref{Simplificationsss},
it suffices to consider $\alpha$ large and $x  \in\mathcal C_\alpha$.
Using polar coordinates with $\beta =  \log\alpha$
for $x\in\mathcal C_\alpha$ and $u\ne 0$,
we write $x=D_s\,\tilde x$ and
$u=D_\sigma\,\tilde u$, with  $s,\sigma\in\R$.
Since $R(x)\ge \beta/2$, \, \cite[Lemma 4.3{\it{(i)}}]{CCS2}
yields
\begin{align*}
|D_{-t}\,u-x| &= |D_{\sigma-t}\,\tilde u-D_s\, \tilde x|
\gtrsim |\tilde x-\tilde u|, \qquad
 x  \in\mathcal C_\alpha.
\end{align*}

By reducing the value of $c$ in the exponential factor in  \eqref{dotK1},
we arrive at
 \begin{equation*}            
|\dot K_t (x,u)| \lesssim
 e^{R(x)}
 \,
\exp
\Big(
-c\,
 |\tilde x-\tilde u|^2\Big)
\,\exp{\left(-c\,|  D_{-t}\,u- x|^2\right)}\,(|D_{-t}\,u| + e^{-ct}), \qquad t>1.
\end{equation*}
An  application of Lemma \ref{preliminary-t-large-not-kernel}  leads to
\begin{equation*}                               
\int_1^\infty
|\dot K_t (x,u)|
dt
 \lesssim  e^{R( x)}
  \,
 \exp\big(- c\,\big|\tilde x-\tilde u\big|^2\big), \qquad
 x  \in\mathcal C_\alpha,
\end{equation*}
and then \eqref{important}  yields
 \begin{align*}                                 
 \|\mathcal H_tf(x)\|_{v(\varrho), [1,\infty)}      \lesssim  e^{R( x)}
  \,\int 
 \exp\big(- c\,\big|\tilde x-\tilde u\big|^2\big) \,|f(u)|\,\misgausskd.
 \end{align*}
Finally, Lemma \ref{bounding-GMAa} implies \eqref{stronger}, and the weak type $(1,1)$  also follows.
\end{proof}

\begin{remark}
Theorem \ref{t>1} says that for the variation in $[1,\infty)$
 the standard weak type $(1,1)$ estimate  is enhanced by a logarithmic factor.
This  phenomenon was observed  in the one-dimensional case in \cite{CCS6}.
\end{remark}

\section{Local versus global region}
\label{local1}

From now on, we shall focus on the variation of $\mathcal H_t$ in the interval $(0,1]$.
This case requires   a further distinction between the local and the global parts of  $\mathcal H_t$.
We  start describing  the localization procedure.

\subsection{\color{black}{Splitting of $\R^n$ into  rings}}\label{subs1}

Let
\begin{align}\label{def:rings}
R_j =
\{x\in \R^n : j\le R(x)\le j+1\}, \qquad j =0,1,\ldots.
\end{align}
These rings cover $\R^n$ and are pairwise disjoint except for boundaries.
We also set $R_j =
\emptyset$  if $j<0$.

 The width  of $R_j$, defined as the $|\cdot|_Q$ distance between the two components of its boundary,  is for
 $j \ge 1$
 \begin{equation}  \label{width}
 \sqrt 2\left(\sqrt{j+1}-\sqrt j\right) =\frac{\sqrt 2}{\sqrt{j+1}+\sqrt j}
  \in \left(\frac1{2\sqrt{j}}, \frac{\sqrt 2}{\sqrt{j}}\right),
 \end{equation}
as easily verified.

We take a sequence of smooth non-negative functions $(r_j)_{j\in\N}$
satisfying $\sum_{j=0}^\infty r_j (x)=1$ for all $x$, with $\text{supp}\,r_j\subseteq R_j\cup R_{j+1}$.
Further, we may choose them such that
\begin{equation}\label{nabla-r}
\big|\nabla r_j (x)\big|\lesssim {1+|x|}.
\end{equation}

We also introduce slightly larger functions  $ \widetilde{r}_j \in\mathcal C_0^\infty (\mathbb R^n)$
taking values in $[0,1]$ such that
 $\widetilde{r_j} = 1 $ in  $\cup_{\nu=j-1}^{j+2} R_{\nu}$  and
$\text{supp}\,\widetilde{r_j}\subset \cup_{\nu=j-2}^{j+3} R_{\nu}$.}
They also satisfy
\begin{equation}\label{nabla2}
\big|\nabla \widetilde{r}_j (x)\big|\lesssim {1+|x|}.
\end{equation}

Then the supports of these   $\widetilde{r_j}$ have bounded overlap,
and
\begin{equation}\label{densityinring}
 e^{-R(x)}\simeq e^{-j} \qquad \text{for} \qquad
 x\in \text{supp}\,\widetilde{r_j}.
\end{equation}

\subsection{The splitting of $\mathcal H_t $}

 We can now split the operator $\mathcal H_t f$ in a local and a global part, in a way adapted to the rings.                       
The local part  is defined by
\begin{align*}                     
    \mathcal H_t^{\mathrm{loc}}  f(x)
    :=& \, \sum_{j=0}^\infty \widetilde r_j (x)
 \mathcal H_t\big(f r_j\big)(x)\\
 = & \, \sum_{j=0}^\infty \widetilde r_j (x)
      \int_{\R^n} K_t(x,u)\,  f(u)\, r_j (u)\,
d\gamma_\infty(u).
\notag
    \end{align*}
This sum is locally finite and thus well defined for any $f \in L^1(\gamma_\infty)$,
because of the bounded overlap of the sets $\text{supp}\:\widetilde{r_j}$.         

Setting
\begin{equation*}                          
 \eta(x,u)=\sum_{j=0}^\infty
\widetilde{r_j} (x) \,r_j (u),
\end{equation*}
we get
\begin{align}\label{kernelHloc}
    \mathcal H_t^{\mathrm{loc}}  f(x)& =   \int_{\R^n} K_t(x,u)\, \eta (x,u)\, f(u)\,
d\gamma_\infty(u),
    \end{align}
and $0 \le \eta(x,u)\le 1$  for all $x,u\in\R^n$.
Moreover,  if $\mathrm{supp}\, f \subset R_j\cup R_{j+1}$, then the support of $\mathcal H_t^{\mathrm{loc}} f $ is contained in $\cup_{\nu=j-2}^{j+4} R_{\nu}$.

The global part of  $\mathcal H_t$ is  defined by $  \mathcal H_t^{\mathrm{glob}} =  \mathcal  H_t -    \mathcal H_t^{\mathrm{loc}}$, or equivalently
\begin{equation*}                        
     \mathcal H_t^{\mathrm{glob}}  f(x)=   \int_{\R^n} K_t(x,u)\, (1- \eta (x,u))\, f(u)\,
d\gamma_\infty(u), \qquad x \in \R^n.
\end{equation*}

\smallskip

 The next two lemmas give  sufficient conditions for the kernel  $K_t(x,u)\, (1- \eta (x,u))$ of the global part to vanish.

 \begin{lemma}  \label{uno}
 Let $j \in \{0,1,\dots\}$.
 If  $x,\,u\in R_j \cup R_{j+1}$, then $\eta(x,u) = 1$.
 \end{lemma}
 \begin{proof}
 When $u\in R_j \cup R_{j+1}$ the function $r_i(u)$ can be nonzero only for $i\in\{j-1,j,j+1\}$, and then
$r_{j-1}(u)+r_j(u)+r_{j+1}(u)=1$.

 Also, for $x\in R_j \cup R_{j+1}$ one has ${\widetilde r}_{j-1} (x)={\widetilde r}_{j} (x)={\widetilde r}_{j+1} (x) = 1$.
This implies that          
$\eta(x,u)=\sum_{\nu =j-1}^{j+1} \widetilde{r_\nu} (x)\, r_\nu(u) =1$.
\end{proof}

  \begin{lemma} \label{due-molt}
If
 \begin{equation*}
 |x-u|_Q\le \frac1{2(1+|x|_Q)},
\end{equation*}
then  $\eta(x,u) = 1$.
 \end{lemma}
 \begin{proof}
Suppose $x\in R_j$. Then $|x|_Q \ge \sqrt{2j}$ and we can write
   \begin{equation*}
 |x-u|_Q\le \frac1{2(1+|x|_Q)} \le  \frac1{2(1+\sqrt{j})} \le \frac1{2\sqrt{j+1}}.
\end{equation*}
  From this and \eqref{width}, we see that $|x-u|_Q$ is less than the widths of $R_{j-1}$ and  $R_{j+1}$ (only that of $R_{j+1}$ if $j=0$ or $j=1$).
  Since   $x \in R_j$,  this means that both $x$ and $u$ must be  in  $ R_{j-1}  \cup  R_j$
   or  $ R_j \cup R_{j+1}$. From Lemma \ref{uno} we then get the assertion.
   \end{proof}

   We will need an estimate for the gradient of $\eta(x,u)$. If a point $(x,u)$ is in the support of $\eta$, then $x$ and $u$ are both in the support of some
$\widetilde{r}_j$. It follows that $1+|u| \simeq 1+|x|$, and \eqref{nabla-r} and \eqref{nabla2} then
imply that
\begin{equation}\label{gradeta}
  |\nabla_x\, \eta(x,u)| +  |\nabla_u\, \eta(x,u)| \lesssim  1+|x|.
\end{equation}

\medskip

\section{The variation of the global part}\label{The variation of the global part}

The result of this section is the following.

 \begin{theorem} \label{t<1,global}
 For each  $\varrho \ge 1$ the operator that maps   $f \in L^1(\gamma_\infty)$ to the function
  \begin{equation*}
  \|{  \mathcal H}_t^{\mathrm{glob}} f(x)\|_{v(\varrho), (0,1]}, \quad x \in \mathbb R^n,
\end{equation*}
is of weak type $(1,1)$ with respect to the measure $\gamma_\infty$.
\end{theorem}

\begin{proof}

   This proof follows the pattern from that of  \cite[Proposition 10.2]{CCS5}. In particular, we will
 need the following result                        
 concerning a maximal operator.

\begin{theorem}\label{stima per M con t piccolo}
The maximal operator defined by
\begin{align*}
S_{0}^{{\rm{glob}}}
 f(x) = \int \sup_{0<t\leq 1}K_t(x,u)\,
\big(1-\eta(x,u)\big) \,|f(u)|\, d\gamma_\infty(u)
\end{align*}
is
 of weak type $(1,1)$ with respect to the invariant measure $\gamma_\infty$.
\end{theorem}

This statement seems to coincide with that of \cite[Theorem 10.1]{CCS5}, but
 our $\eta(x,u)$ is not the same as that in \cite{CCS5}. However, by tracing the proof in \cite{CCS5},
one sees that what matters is only the implication
  \begin{equation*}            
  \eta(x,u)<1 \qquad    \Rightarrow    \qquad
 |x-u|_Q > \frac1{2(1+|x|_Q)},
\end{equation*}
which is a consequence of our Lemma \ref{due-molt}. In this way, Theorem \ref{stima per M con t piccolo} follows.

 \vskip5pt

   We will also need
   \cite[Proposition 9.1]{CCS5}. It says
 that for any $(x,u)\in \mathbb R^n\times\mathbb R^n$,
the number of zeros $N(x,u)$ in $(0,1]$ of the map $t \mapsto \dot K_t(x,u)$ is bounded
by a positive integer  $\bar N$ that depends only on $n$ and $B$.

As in \cite[Proposition 10.2]{CCS5},
we denote these zeros by $t_1(x,u)< \cdots< t_{N(x,u)}(x,u)$,
and set $t_0(x,u) =0$, \hskip4pt
$t_{N(x,u)+1}(x,u) =1$.
Since $K_t(x,u)$ vanishes at $t=0$ if $x\ne u$,
it follows from the fundamental theorem of calculus that 
\begin{multline*}
\int_0^1 \left|\dot K_t(x,u)\right| dt
= \sum_{i=0}^{N(x,u)} \left|\int_{t_i(x,u)}^{t_{i+1}(x,u)} \dot K_t(x,u) dt\right|
\\
= \sum_{i=0}^{N(x,u)} \left|K_{t_{i+1}(x,u)}(x,u) - K_{t_i(x,u)}(x,u)\right|
\\
\leq 2
\sum_{i=0}^{N(x,u)+1}K_{t_i(x,u)}(x,u)
\;\lesssim \;\bar N
\sup_{0<t\leq 1}K_{t}(x,u).
\end{multline*}

By \eqref{important}, which remains valid with the extra factor $1-\eta(x,u)$, we obtain
 \begin{equation*}                   
\| \mathcal H_t^{\mathrm{glob}} f(x)\|_{v(\varrho), (0,1]} \lesssim
\int     \sup_{0<t\leq 1}K_t(x,u)\,(1-\eta(x,u)\, |f(u)|\,d\gamma_\infty(u).
 \end{equation*}
 Now Theorem \ref{stima per M con t piccolo} implies the assertion of Theorem \ref{t<1,global}.
     \end{proof}

\vskip15pt

\section{The local part for small $t$}\label{tsmall}
In this section, we prove the following result.          

   \begin{theorem}  \label{locsmallt}
     For each  $\varrho > 2$ the operator that maps   $f \in L^1(\gamma_\infty)$ to the function
  \begin{equation*}
  \|\mathcal H_t^{\mathrm{loc}} f(x)\|_{v(\varrho),(0,1]}, \quad x \in \mathbb R^n,
\end{equation*}
is of weak type $(1,1)$ with respect to the measure $\gamma_\infty$.
   \end{theorem}

The proof
relies on   vector-valued Calder\'on--Zygmund theory.

\subsection{Preparations}
Part of our notation in this section follows that  of  \cite{Crescimbeni}.
Let  $\Theta$ be the set of all finite, increasing sequences in $(0,1]$ written
\begin{equation*}
 \succvarepsilon=(\varepsilon_i)_0^N
\end{equation*}
for some $ N = N(\succvarepsilon)\in \N_+.${\Blue{
}}
Then we let $F$ be the vector space of all functions
\begin{equation*}
 g: \{(i,\succvarepsilon) \in \N \times \Theta: 1 \le i \le N(\succvarepsilon)\} \to \R.
\end{equation*}
For $\varrho \in [1,\infty)$ we define $ F_\varrho$ as that subspace of $ F$ consisting of those $g $ for which the mixed norm
\begin{equation}
 \|g\|_{ F_\varrho}:=
 \sup_{\succvarepsilon \in\Theta}
 \left(\sum_{i=1}^{N(\succvarepsilon)}
 |g(i,\succvarepsilon)|^\varrho\right)^{1/\varrho}
\end{equation}
    is finite.     Then  $ F_\varrho$
     is a normed space and a Banach space.

We shall work with the operator $V$ that maps real-valued functions $f\in L^1(\gamma_\infty)$ to $F$-valued functions defined in $\R^n$
and is given by
\begin{align*}
 Vf(x)(i, \succvarepsilon) = \mathcal H_{\varepsilon_{i}}f(x)-
\mathcal H_{\varepsilon_{i-1}}f(x),  \qquad i = 1,\dots,N(\succvarepsilon ), \;\; \succvarepsilon \in \Theta.
 \end{align*}
      Similarly, we set
 \begin{align} \label{Vloc}
 V^{\text{loc}}f(x)(i, \succvarepsilon) = \mathcal H^{\text{loc}}_{\varepsilon_{i}}f(x)-
\mathcal H^{\text{loc}}_{\varepsilon_{i-1}}f(x),  \qquad i = 1,\dots,N(\succvarepsilon),
 \end{align}
 and analogously for $V^{\text{glob}}f$.

Then for  $\varrho \in [1,\infty)$
\begin{equation}  \label{equivalence}
 \|Vf(x)\|_{F_\varrho}=\| \mathcal H_t f(x)\|_{v(\varrho), (0,1]},
\end{equation}
and similar equalities hold with superscripts $\text{loc}$ or $\text{glob}$.

Theorem \ref{locsmallt} can now be rewritten in the following equivalent way.
 \begin{theorem} \label{locsmallt_v2}
     For each  $\varrho > 2$ the operator that maps   $f \in L^1(\gamma_\infty)$ to the function
\begin{equation*}
x\mapsto\|V^{\rm{loc}}f(x)\|_{F_\varrho}\,, \qquad x\in\R^n,
\end{equation*}
is of weak type $(1,1)$ with respect to the measure $\gamma_\infty$.
  \end{theorem}
  The advantage of Theorem
\ref{locsmallt_v2} is that  the proof  can be based on vector-valued Calder\'on--Zygmund theory.
To apply this machinery, we will pass to Lebesgue measure,
and set
\begin{equation}\label{deftildeV:oper}
\widetilde V^{\text{loc}}\,g(x)=
 e^{-R(x)}\,
V^{\text{loc}}\big ( g(\cdot)\, e^{R(\cdot)}\,\big)(x).
\end{equation}
Lebesgue measure will be written either $dx$ or $du$.

The following proposition            
clarifies  the connection between
$\widetilde V^{\text{loc}}$ and $V^{\text{loc}}$.

\begin{proposition}\label{prop81}
Let $\varrho>2$. If
the operator that maps   $g \in L^1(du)$ to the function
  \begin{equation*}
 x\mapsto \left\|\widetilde V^{\rm{loc}}\,g(x)\right\|_{F_\varrho}, \quad x \in \mathbb R^n,
\end{equation*}
is of weak type $(1,1)$ with respect to   Lebesgue measure,
then the operator that maps   $f \in L^1(\gamma_\infty)$ to the function
  \begin{equation*}
x\mapsto   \|V^{\rm{loc}}f(x)\|_{F_\varrho}\,, \quad x \in \mathbb R^n,
   \end{equation*}
 is of   weak type $(1,1)$ with respect to  $\gamma_\infty$.
\end{proposition}

\begin{proof}
  We have
   \begin{align*}                        
   \|V^{\text{loc}}f(x)\|_{F_\varrho}
   = &\, \left\| \sum_{j=0}^\infty \widetilde {r_j} (x)\, V (fr_j)(x)\right\|_{F_\varrho}
  \\  \le & \,\left\|\sum_{j=0}^\infty \widetilde r_j(x)\,V^{\text{glob}} (fr_j)(x)\right\|_{F_\varrho} +
  \left\|  \sum_{j=0}^\infty \widetilde r_j(x)\,V^{\text{loc}} (fr_j)(x)\right\|_{F_\varrho} = I_{\text{glob}} + I_{\text{loc}}.
   \end{align*}
 Since the $\widetilde {r_j}$ have supports with bounded overlap,
  the  sums here are uniformly locally finite. It follows that
    \begin{align*}
 \|  I_{\text{glob}} \|_{L^{1,\infty}(\gamma_\infty)} \lesssim  \sum_{j=0}^\infty \left\|\widetilde r_j(x)\,\|V^{\text{glob}} (fr_j)(x)\|_{F_\varrho}\, \right\|_{L^{1,\infty}(\gamma_\infty)}.
   \end{align*}
   The $F_\varrho$ norm here equals $\| \mathcal H_t^{\text{glob}} (fr_j)(x)\|_{v(\varrho), (0,1]}$. After estimating the factor $\widetilde r_j(x)$  by 1, we can apply Theorem
   \ref{t<1,global} and get
    \begin{align*}
 \|  I_{\text{glob}} \|_{L^{1,\infty}(\gamma_\infty)}  \lesssim   \sum_{j=0}^\infty \left\|fr_j \right\|_{L^{1}(\gamma_\infty)} \simeq  \|f\|_{L^{1}(\gamma_\infty)}.
   \end{align*}

   We  estimate  the $L^{1,\infty}(\gamma_\infty)$ quasinorm of $II$  similarly, but then apply the definition of
   $\widetilde V^{\text{loc}}$. This gives
    \begin{align*}
 \| I_{\text{loc}} \|_{L^{1,\infty}(\gamma_\infty)}  \lesssim  &   \sum_{j=0}^\infty  \left\| \widetilde r_j(x)\,\|V^{\text{loc}} (fr_j)(x)\|_{F_\varrho} \right\|_{L^{1,\infty}(\gamma_\infty)} \\
\simeq &
   \sum_{j=0}^\infty    \left\| \widetilde r_j(x)\,e^{R(x)}\left\|\widetilde V^{\text{loc}} (fr_je^{-R(\cdot)})(x)\right\|_{F_\varrho}  \right\|_{L^{1,\infty}(\gamma_\infty)}   \\
    \simeq  & \sum_{j=0}^\infty    \left\| \widetilde r_j(x)\,\left\|\widetilde V^{\text{loc}} (fr_je^{-R(\cdot)})(x)\right\|_{F_\varrho}  \right\|_{L^{1,\infty}(dx)};
      \end{align*}
  in the last step here, we  switched
    from $\gamma_\infty$ to Lebesgue measure. This  was possible because of \eqref{densityinring}. Next, we estimate  $\widetilde r_j(x)$ by 1 and use the hypothesis of the proposition, getting
    \begin{align*}
  \| I_{\text{loc}} \|_{L^{1,\infty}(\gamma_\infty)}  \lesssim    \sum_{j=0}^\infty \left\|fr_j e^{-R(\cdot)}\right\|_{L^{1}(du)} \simeq \|f\|_{L^{1}(\gamma_\infty)}.
   \end{align*}
   The proposition is proved.
   \end{proof}

\subsection{Vector-valued Calder\'on--Zygmund operators}\label{CalderonZygmund}
This theory will be applied to                               
the operator  $\widetilde V^{\text{loc}}$.  The following proposition gives the strong  $(p,p)$ bound needed.

\begin{proposition}\label{prp2}
 For any  $\varrho>2$ and any $1<p<\infty$, the operator that maps  $g \in L^p(du)$ to the function
 \begin{equation*}
x\mapsto   \left\|\widetilde V^{\text{loc}}\,g(x)\right\|_{F_\varrho}, \qquad x \in \mathbb R^n,
  \end{equation*}
 is bounded from $L^p(du)$ to $L^{p} (dx )$.
\end{proposition}
\begin{proof}

As a consequence of the definitions of $\widetilde V^{\text{loc}}$  and $ V^{\text{loc}}$ and the bounded overlap of the $\widetilde{r_j}$,
 one has
  \begin{align*}               
  \left\|\,   \left\|\widetilde V^{\text{loc}}\,g(x)\right\|_{F_\varrho}\,\right\|_{L^p(dx)}^p
 & =
 \int   \left\|
 \sum_{j=0}^\infty
 \widetilde r_j (x) \,e^{-R(x)}\,
V\left(
g \,r_j\, e^{R(\cdot)} \right)(x)\right\|_{F_\varrho}^p \,dx\\
&\lesssim
  \sum_{j=0}^\infty
\int     \Big\| \widetilde r_j (x) e^{-R(x)}\,
 V\left ( g\, r_j \,e^{R(\cdot) }\right) (x)
\Big\|_{F_\varrho}^p\, dx\\
&\simeq \sum_{j=0}^\infty  e^{-jp} \int
  \Big\|
 \widetilde{r}_j (x)\,
V\left ( g\, r_j \,e^{R(\cdot) }\right) (x)
\Big\|_{F_\varrho}^p \,dx\\
&\simeq
 \sum_{j=0}^\infty  e^{j(1-p)}  \int
  \Big\|
 \widetilde{r}_j (x)\,
V\left ( g\, r_j \,e^{R(\cdot) }\right) (x)
\Big\|_{F_\varrho}^p \,d\gamma_\infty(x).
\end{align*}
Notice that  we also used \eqref{densityinring} here.
 In the last expression, we delete the factor $\widetilde r_j (x)$ and observe that the resulting $F_\varrho$ seminorm  equals
$ \left\|\mathcal H_t \left ( g\, r_j \,e^{R(\cdot) }\right) (x)\right\|_{v(\varrho), (0,1]}$.

As mentioned in the introduction, the variation operator for  $\mathcal H_t$ is  bounded on $L^p(\gamma_\infty)$ with    $1<p<\infty$, for   $\varrho > 2$ \cite{Le Merdy}. We therefore obtain
   \begin{align*}
   \left\|\,   \left\|\widetilde V^{\text{loc}}\,g(x)\right\|_{F_\varrho}\,\right\|_{L^p(dx)}^p
 & \lesssim   \sum_{j=0}^\infty  e^{j(1-p)}  \int
  \left\|\mathcal H_t \left( g\, r_j \,e^{R(\cdot) }\right) (x)
  \right\|_{v(\varrho), (0,1]}^p d\gamma_\infty(x)
    \\ & \lesssim
   \sum_{j=0}^\infty  e^{j(1-p)}\left\| g\, r_j \,e^{R(\cdot)} \right\|_{L^p(\gamma_\infty)}^p
  \simeq   \sum_{j=0}^\infty  \| g\, r_j \, \|_{L^p(du)}^p  \le  \big\|g\big\|^p_{L^p(du)},
  \end{align*}
  proving the assertion.
  \end{proof}

\vskip5pt

We shall now prove the appropriate standard estimates for the vector-valued kernel  of
$\widetilde V^{\text{loc}}$.

From    \eqref{Vloc} and  \eqref{kernelHloc}, we see that
$V^{\text{loc}}$  is an integral operator with an  $F$-valued
 kernel  $ M^{\text{loc}} (x,u)$  given by
 \begin{equation*}                               
 M^{\text{loc}}(x,u) (i,\succvarepsilon) =\big(
K_{\varepsilon_{i}}(x,u)-K_{\varepsilon_{i-1}} (x,u)\big)\, \eta(x,u), \qquad
i = 1,\dots,N(\succvarepsilon).
\end{equation*}
 This means that  for  $f \in L^1(\gamma_\infty)$
     \begin{equation}\label{formula1_Bochner}
V^{\text{loc}} f(x)=\int  M^{\text{loc}} (x,u) f(u)\,d\gamma_\infty (u).
\end{equation}
We will need the
 kernel $\widetilde M^{\rm{loc}} (x,u)$ of $\widetilde V^{\text{loc}}$, for integration against Lebesgue measure in the sense that
 \begin{equation}\label{formula2_Bochner}
\widetilde V^{\text{loc}}\,f(x)=
\int \widetilde M^{\rm{loc}}(x,u)\,f(u)\,du
\end{equation}
for suitable $f$.

\begin{remark}
Notice that formulae \eqref{formula1_Bochner} and \eqref{formula2_Bochner} make sense  in  $F$ for all $x\in\R^n$, since the integrals may be computed
coordinatewise, that is, at one $(i,\epsilon) $ at a time. 
In the Appendix (see Section \ref{Appendix}) it will be shown that
for   $x\notin {\text{supp}}\,f$ 
the integral in  \eqref{formula2_Bochner} is an  $F_\varrho$-Bochner integral, and thus 
$\widetilde V^{\text{loc}}\,f(x) \in F_\varrho$.
\end{remark}

From \eqref{deftildeV:oper} we get
\begin{multline}\label{def:Pcalli}
\widetilde M^{\rm{loc}} (x,u)(i,\succvarepsilon)=
 e^{-R(x)}\,
 M^{{\rm{loc}}} (x,u)(i,\succvarepsilon)
  = e^{-R(x)}\,
\big(K_{\varepsilon_{i+1}}(x,u)-K_{\varepsilon_{i}}(x,u)\big)\,\eta(x,u), \\ \qquad\qquad
i = 1,\dots,N(\succvarepsilon).
\end{multline}
In analogy with \eqref{equivalence}, this implies
\begin{align} \label{equiv}
\|\widetilde M^{\rm{loc}}(x,y) \|_{F_\varrho}
&=\,e^{-R(x)}\,
\|K_t(x,u) \,\eta(x,u)\|_{v(\varrho), (0,1]}\notag
 \\&= e^{-R(x)} \,\eta(x,u)\,
\|K_t(x,u)\|_{v(\varrho), (0,1]},
\end{align}
where the variation is meant with respect to $t$.

We will need an auxiliary result.
\begin{proposition}                                            
\label{lemma-integral_dt}
Let $p,\: r\ge 0$ with $p+ r/2 > 1 $. Assume that $\eta(x,u)\neq 0$  and  $x\neq u$. Then for $\delta>0$
\begin{equation*}                 
\int_0^{1}
t^{-p}   \exp\left(-\delta\,\frac{|u- D_t \,x |^2}t\right) |x|^{r} \,  dt\lesssim C\,{|u-x|^{-2p-r+2}}.
\end{equation*}
Here the constant $C$ may depend on  $\delta,\: p$ and   $r$, in addition to   $n$, $Q$ and  $B$.
\end{proposition}

\begin{proof}
 The statement of this proposition is similar to that of  \cite[Proposition 8.3]{CCS5}, which is based on
\cite[Lemma 8.1]{CCS3}. But our function $\eta$ is not the same as there, and we must verify that
for  $\eta(x,u)\neq 0$ and  $0<t<1$
\begin{equation}\label{occorre}
  \frac{|D_t\,x - u|^2}{t} > c \,\frac{|x - u|^2}{t} +c t |x|^2-C.
\end{equation}
Once this is established, we can follow the argument for \cite[Lemma 8.1]{CCS3} and finish the proof.

 To verify \eqref{occorre}, assume first that
 $t < c_0 |x-u|/|x|$ for a small $c_0 >0$ to be chosen. Then \cite[Lemma~2.3]{CCS3} implies
$|D_t\,x - x| \le C t |x| < Cc_0 |x-u| $ and thus
\begin{equation*}
  |D_t\,x - u| \ge |x - u| -  |D_t\,x - x| > |x - u| - Cc_0 |x-u|.
\end{equation*}
So with $c_0 >0$ small enough we get
\begin{equation*}
  |D_t\,x - u| \, \ge \frac12
  \, |x - u| \,
  \gtrsim\, |x - u| +t |x|.
\end{equation*}
To obtain \eqref{occorre} in this case, take squares and divide by $t$.

In the opposite case $c_0 |x-u|/|x| \le t < 1$, we use \cite[Lemma 4.1]{CCS2} to conclude that
\begin{equation*}
\frac  \partial{\partial t}\, | D_t\,x|_Q = \frac  \partial{\partial t}\, \sqrt{2R( D_t\,x)}
\simeq | D_t\,x|_Q  \simeq | x|_Q.
\end{equation*}
Integration yields
\begin{equation*}
| D_t\,x|_Q  - | x|_Q  \simeq t | x|_Q\,
\gtrsim\, t | x|_Q + |x-u|.
\end{equation*}
Thus
\begin{align}
|D_t\,x - u|_Q &\ge | D_t\,x|_Q  - | u|_Q  = | D_t\,x|_Q  - | x|_Q  +  | x|_Q -  | u|_Q\notag
\\
&\ge ct | x|_Q + c|x-u|  - \big|| x|_Q  -  | u|_Q\big|.\label{nuovo}
\end{align}
Since  $\eta(x,u) > 0$,  there exists a  $j \in \mathbb N $ such that  such that $\widetilde r_j (x) > 0$ and  $ r_j (u) > 0$.
  Thus   $\big|| x|_Q  -  | u|_Q\big|
  \lesssim   1/\left(1+\sqrt j\right)  \simeq  \frac 1{1 +|x|}$,
 and by squaring \eqref{nuovo} we get
\begin{align*}
|D_t\,x - u|_Q^2 \ge &\, ct^2 | x|_Q^2 + c|x-u|^2 - C(t | x|_Q +|x-u|) /(1+|x|)\\
 \ge &\, ct^2 | x|_Q^2 +c|x-u|^2 - Ct.                                                  
\end{align*}
This implies \eqref{occorre}.
\end{proof}

\vskip7pt

\begin{proposition}\label{lemma-Calderon-no}
(\rm a)
For all $(x,u)$ such that $x\neq u$            
  $$
\|\widetilde M^{\rm{loc}}(x,u) \|_{F_\varrho}\lesssim \frac 1{|x-u|^{n}} .
$$
  \medskip

(\rm b)
If the points $x,\:u$  and $u'$ satisfy $|x-u| > 2  |u-u'|$, then
 $$
 \| \widetilde M^{\rm{loc}}(x,u) - \widetilde M^{\rm{loc}}(x,u')\|_{F_\varrho}
  \lesssim \frac{|u-u'|}{ |x-u|^{n+1}}.
  $$

 
The implicit constants in this proposition depend only on $n$, $B$ and $Q$.
\end{proposition}

\begin{proof}
  \hskip4pt  (a) \hskip4pt
Starting from \eqref{equiv}, we use   \eqref{var} and then \eqref{dotKeps} to get
\begin{align*}                   
  \|\widetilde M^{\rm{loc}}(x,u) \|_{F_\varrho}
& =     e^{-R(x)}\, \eta(x,u)\, \| \ K_t(x,u)\|_{v(\varrho), (0,1]}
 \le e^{-R(x)}\, \int_0^1 \big| \dot K_t(x,u)\big|\,dt \\
 & \lesssim
 \int_0^1 t^{-\frac n2} \exp\left(- c\, \frac{|u-D_t\,x |^2}t\right)\,
\left(\frac{1}{t}
+\frac{|x|}{\sqrt t}
\right)\, dt.
\end{align*}

Since  $\eta(x,u)\neq 0$,
Proposition
\ref{lemma-integral_dt}
 tells us that the last  integral  is bounded
 by $ C |u-x|^{-n}$,
and (a) follows.

\medskip

\noindent (b) \hskip4pt
We start almost as in (a), and find
 \begin{align*}
 \| \widetilde M^{\rm{loc}}(x,u) - \widetilde M^{\rm{loc}}(x,u')\|_{F_\varrho} & =
    e^{-R(x)}\,  \| \eta(x,u)\, K_t(x,u) - \eta(x,u')\, K_t(x,u')\|_{v(\varrho), (0,1]} \\
   & \le   e^{-R(x)}\,  |\eta(x,u) - \eta(x,u')| \,\|\ K_t(x,u)\|_{v(\varrho), (0,1]} \\& +
   e^{-R(x)}\,    \eta(x,u') \,\| \ K_t(x,u) - K_t(x,u')\|_{v(\varrho), (0,1]} =: \mathrm{I} + \mathrm{II}.
  \end{align*}
 
 To estimate the term I we use  \eqref{gradeta} to get $|\eta(x,u) - \eta(x,u')| \lesssim (1+|x|)\,|u-u'| $ and
  estimate $\| K_t(x,u)\|_{v(\varrho), (0,1]}$ precisely as in (a). If $|x| > 1$,  Proposition
\ref{lemma-integral_dt} then leads to 
 \begin{align} \label{termI}
|I|  \lesssim |u-u'|\,|u-x|^{-n-1}.
 \end{align}
 In the opposite case $|x| \le 1$, we get 
$|\mathrm{I}|  \lesssim |u-u'|\,|u-x|^{-n}$ in the same way. But in that case $|u-x| \lesssim 1$ since $\eta(x,u)$ or $\eta(x,u')$ is nonzero, so \eqref{termI} follows again.

 For II  we estimate $\eta(x,u')$ by 1 and let  $w = u-u'$.   We can apply  \eqref{var} to write 
 \begin{align}\label{termII}
 \| \ K_t(x,u) - K_t(x,u')\|_{v(\varrho), (0,1]} &\le  \int_0^1 \big| \dot K_t(x,u) - \dot K_t(x,u')\big|\,dt \\
& \le  \int_0^1  \left|\int_0^1  \langle \nabla_2\dot K_t(x,u'+sw),w \rangle\,ds\right|\,dt,
  \end{align}
  where $\nabla_2$ denotes the gradient with respect to the second variable.
  We need an estimate of the second derivative  $\nabla_2\dot K_t$\,, to begin with at the point $(x,u)$.
  
  Let for simplicity $\ell \in \{1,\dots,n\}$.
   From \cite[Lemma 4.1]{CCS3} we have
\begin{align*}
\partial_{u_\ell} K_t(x,u)   =
- K_t(x,u) \,R_\ell(t,x,u),\qquad t>0,
\end{align*}
where
\begin{align*}
R_\ell(t,x,u) = \left\langle Q_t^{-1} e^{tB} \, ( D_{-t}\, u - x), \,e_\ell\right\rangle.
\end{align*}
Further
\begin{equation} \label{Rell}
|R_\ell(t,x,u) |\lesssim
 \frac{|D_{-t}\,u-x|}t \simeq
 \frac{|u-D_t\, x|}t,
\end{equation}
since $t\in (0,1)$,
 and                  
\begin{align*}
 & \dot R_\ell(t,x,u) =
  {{-\left\langle Q_t^{-1}\, e^{tB}\, Q\, e^{tB^*}\, Q_t^{-1} e^{tB} \, ( D_{-t}\, u - x), \,e_\ell\right\rangle}}\\
  &\qquad\qquad+{{\left\langle Q_t^{-1} Be^{tB} \, ( D_{-t}\, u - x), \,e_\ell\right\rangle}}+{{\left\langle Q_t^{-1} e^{tB} \, Q_\infty\, B^* \, Q_\infty^{-1}\,  D_{-t}\, u , \,e_\ell\right\rangle}}.
\end{align*}
In the last term here, we write $D_{-t}\, u = (D_{-t}\, u - x) +x$, and it  follows that
\begin{align} \label{dotRell}
| &\dot R_\ell(t,x,u)|  \lesssim
\frac{| u-D_t\,x|}{t^2}+\frac{| D_{-t}\, u  -x|+|x|}{t}\lesssim \frac{| u-D_t\,x|}{t^2}+\frac{|x|}t.
  \end{align}
  
  Clearly,
 \begin{align*}
 \partial_t\,\partial_{u_\ell} K_t(x,u) &  =  \dot K_t(x,u)\,R_\ell(t,x,u) +
 K_t(x,u) \, \dot R_\ell(t,x,u),
\end{align*}
 and to estimate the two terms here, we combine  \eqref{dotKeps} with \eqref{Rell} and 
  \eqref{litet} with  \eqref{dotRell}. As a result,
  \begin{align*}
 |\partial_t\,\partial_{u_\ell} K_t(x,u)| & \lesssim e^{-R(x)}\, t^{-\frac n2} \exp\left(- c\, \frac{|u-D_t\,x|^2}t\right)\,
 \left(\frac{1}{t}+\frac{|x|}{\sqrt t}\right)\, \frac{|u-D_t\, x|}t \\ &+
  e^{-R(x)}\, t^{-\frac n2} \exp\left(- c\, \frac{|u-D_t\,x|^2}t\right)\,
 \left( \frac{| u-D_t\,x|}{t^2}+\frac{|x|}t\right)
\end{align*}
  
   By reducing the constants $c$ in the factors
$\exp\left(- c\, {|u - D_t\,x|^2}/t\right)$ ,
 we can eliminate factors ${|D_t\,x -u|}/\sqrt t$ and get
  \begin{align*}
 |\partial_t\,\partial_{u_\ell} K_t(x,u)| & \lesssim e^{-R(x)}\, t^{-\frac n2} 
 \exp\left(- c\, \frac{|u-D_t\,x|^2}t\right)\,
 \left(\frac{1}{t^{3/2}}+\frac{|x|}{ t}\right).
\end{align*}

  Proposition \ref{lemma-integral_dt} now implies
 \begin{align}\label{dtdu}
\int_{0}^{1} |\partial_t\,\partial_{u_\ell} K_t(x,u)| \,dt \lesssim |u-x|^{-n-1}.
\end{align}
  The next step is to observe that the arguments leading to this estimate remain valid if $u$ is replaced by
  $u'+sw$ for any $s \in [0,1]$, in the left-hand side. For the right-hand side this replacement does not change the order of magnitude, so we can keep $|u-x|$ there.
  
  Swapping the order of integration in \eqref{termII} and observing that $|w| = |u-u'|$, we  conclude from this that
  \begin{equation*}
    |\mathrm{II}| \lesssim \int_{0}^{1}  \frac{|u-u'|}{ |x-u|^{n+1}}\,ds = \frac{|u-u'|}{ |x-u|^{n+1}}.
  \end{equation*}
  
Together with \eqref{termI}, this estimate  completes the proof of (b) and that of Proposition~\ref{lemma-Calderon-no}.
    \end{proof}
  
 \vskip30pt


We now combine Propositions  \ref{prp2} and   \ref{lemma-Calderon-no} with  the Bochner integral representation in
 \eqref{formula2_Bochner} for  $x\notin \mathrm{supp}\, f$.  This is all we  need to apply the standard Calderón-Zygmund argument with values in $F_{\varrho}$, to obtain the weak type (1,1) estimate in  Theorem \ref{locsmallt_v2}.
  Theorem   \ref{locsmallt} also follows.

  Finally, Theorem \ref{thm} follows from
Theorems  \ref{t>1}, \ref{t<1,global} and  \ref{locsmallt}.

\section{A counterexample}\label{counterexample}

\begin{theorem}\label{thm_counterex}
The variation operator
  \begin{equation*}
    f \mapsto \|\mathcal{H}_t f (x)\|_{v(2),\R_+}, \qquad x\in \R^n,
  \end{equation*}
  is not of strong nor weak type $(p,p)$ with respect to $\gamma_\infty$,  for any $p \in [1,\infty)$.
\end{theorem}

\smallskip

This theorem shows that  the condition $\varrho>2$ is necessary in Theorem \ref{locsmallt} and therefore also in Theorem \ref{thm}. The proof    of  Theorem \ref{thm_counterex}
  falls naturally into several parts and will occupy the entire section.

\medskip

We first describe the plan of the proof. 
It is based on a probabilistic result for the one-dimensional torus  due to Qian  \cite{Qian}.
To relate our setting to that of Qian, we first approximate $\mathcal{H}_t$ by a convolution operator $\mathcal{H}_t^c$. Staying in a compact set in $\R^n$, we consider only  $t \in (0,1]$ and verify that this approximation is admissible. The kernel $\mathcal{H}_t^c$ is then diagonalized, which gives a tensor product of one-dimensional convolution operators. 
Our focus will be on the first coordinate, and there  we apply the convolution operator to a  function defined as a sum of Rademacher functions. We also move from a compact interval to the one-dimensional torus, where we will have a convolution operator acting on a similar Rademacher sum. The parameter $t$ is now restricted to a dyadic sequence. 
Next, we  verify that we can replace the operator by a simpler mean value operator on the torus, which is then replaced by a dyadic conditional expectation operator, to which Qian's result applies. This gives the counterexample that proves Theorem \ref{thm_counterex}.

  \subsection{Approximation of the kernel}                
 We  consider only $t \in (0,1]$, and  $f$ will be supported in a  compact set. Moreover, $\mathcal{H}_t f$ will only be considered at points $x$ in 
  another compact set.
     For the $L^p$  norms of  $f$ and those of the variation of $\mathcal{H}_t f$, we can therefore use Lebesgue measure instead of $\gamma_\infty$. In the integral  \eqref{def-int-ker} defining  $\mathcal{H}_t f(x)$, we can  write $du$ instead of $d\gamma_\infty(u)$.
  This also allows us to delete the factors     ${(\det Q_\infty)}^{{1}/{2}}$      and $e^{R(x)}$ in the expression
  \eqref{mehler} for the Mehler kernel. What remains is the kernel
  \begin{align}    \label{mehlersimpl}                   
\widetilde K_t (x,u)
=
{(\det \, Q_t)}^{-{1}/{2} }\,
\exp \Big[
{-\frac12
\left\langle \left(
Q_t^{-1}-Q_\infty^{-1}\right) (u-D_t \,x) \,, u- D_t\, x\right\rangle}\Big]\!,
\end{align}
and we will approximate this kernel.
The  $\mathcal O(\cdot)$  symbol will be used for
scalars,   vectors and matrices and is for $t \to 0$. The implicit constants involved may depend on the compact sets mentioned, as well as on $Q$ and $B$.

 Since $t \le 1$, the definition \eqref{defQt} implies  $Q_t = Q t+ \mathcal O(t^2)$, and then
\begin{equation*}
  Q_t^{-1}-Q_\infty^{-1} = Q^{-1} t^{-1} \big(1+ \mathcal O(t)\big).
\end{equation*}
   Further,   $D_t\, x = x +  \mathcal O(t)$ because of \eqref{def:Dt}. Since  $x$ and $u$ stay bounded, it follows that
      \begin{align}\label{diff:section9}
&\left\langle \left(
Q_t^{-1}-Q_\infty^{-1}\right) (u-D_t \,x) \,, u-D_t\, x\right\rangle
\notag\\
&=
\left\langle
 Q^{-1} t^{-1} (u-x +  \mathcal O(t)),
u-x +  \mathcal O(t)
\right\rangle
+\mathcal O\big(|x-u|+t\big)
\notag\\
&=
\left\langle
 Q^{-1} t^{-1} (u-x),
u-x
\right\rangle+
 \mathcal O\big(|x-u|+t\big)
 .
\end{align}
We also observe that $\det \, Q_t =t^n\, \det  Q\, (1+ \mathcal O(t))$, so that
 \begin{equation}\label{det}
   {(\det \, Q_t)}^{-{1}/{2} } = {(\det \, Q)}^{-{1}/{2} }\,t^{-{n}/{2}} (1+ \mathcal O(t)).
 \end{equation}

This makes it natural to approximate $\widetilde K_t (x,u)$ by the simple convolution kernel 
$K^{\mathrm c}_t (x-u)$, where
\begin{equation}\label{ktc}
  K^{\mathrm c}_t (y) = {(\det \, Q)}^{-{1}/{2} }\,t^{-{n}/{2}} \,\exp \Big(
{-\frac12\, t^{-1}
\left| Q^{-1/2} y \right|^2}\Big),
\end{equation}
and we let 
  \begin{align*}                      
 \mathcal H^{\mathrm c}_tf(x) = K^{\mathrm c}_t * f(x).
\end{align*}
  

\subsection{The difference operator}

Let $\Delta_t$ be the operator defined by the kernel $\widetilde K_t(x,u)- K^{\mathrm c}_t (x-u)$.

\begin{proposition}\label{lemmaIIt}
 Let $C_1$ and  $C_2$ be compact  subsets of $\R^n$. If $f \in L^2(\R^n)$ has support contained in
$C_1$, then
\begin{equation*}
 \left\| \left\|\Delta_{2^{-2\ell}} f \right\|_{v(2),\N_+} \right\|_{L^2(C_2)} \lesssim \|f\|_{L^2(C_1)}.
 \end{equation*}
 Here the implicit constants may depend on $C_1$ and  $C_2$, in addition to $Q$ and $B$.
\end{proposition}

\begin{proof}
By means  of \eqref{diff:section9} and \eqref{det}, we get for  $x \in C_2$ and  $u \in C_1$
\begin{align*}
&\widetilde K_t(x,u)\\
=&\,
(\det \, Q)^{-{1}/{2}}\,t^{-{n}/{2} }(1+\mathcal O(t))\,
\exp\Big( -\frac12  t^{-1}
\left\langle Q^{-1} (u-x), u-x \right\rangle
+ \mathcal O\big(|x-u|+t\big)
\Big)\\
= &\, K^{\mathrm c}_t (x-u)\,(1+\mathcal O(t))\,\exp\Big(\mathcal O\big(|x-u|+t\big)
\Big)\\
=&\, K^{\mathrm c}_t (x-u) + K^{\mathrm c}_t (x-u)\,\mathcal O\big(|x-u|+t\big).
\end{align*}
In the last term here, we may replace $|x-u|$ by $t^{1/2}$, if we reduce the factor $1/2$
in the exponential factor in the expression  \eqref{ktc} for $K^{\mathrm c}_t $. Thus
\begin{align*}
&|\widetilde K_t(x,u) -  K^{\mathrm c}_t (x-u)|  \lesssim
t^{1/2}\, t^{-n/2}\,\exp\Big(  -c  t^{-1}
|Q^{-1/2} (x-u)|^2
\Big).
\end{align*}
With $f $  supported  in $C_1$ and $x \in C_2$, we thus have

\begin{equation*}
  |\Delta_t f(x)| \lesssim t^{1/2}\, \int f(u)\, t^{-n/2}\,\exp\Big(  -c  t^{-1}
|
 Q^{-1/2} (x-u)
|^2\Big)\,du.
\end{equation*}
The integral here is the convolution of $f$ and an integrable kernel, and hence
\begin{equation*}
  \|\Delta_t f\|_{L^2(C_2)} \lesssim  t^{1/2}\, \|f\|_{L^2(C_1)}.
\end{equation*}
Letting  $t = 2^{-2\ell}$, squaring and summing, we obtain
\begin{equation*}
  \left\|\left(\sum_{\ell =1}^{\infty}(\Delta_{2^{-2\ell}} f)^2\right)^{1/2}\right\|_{L^2(C_2)} \lesssim  \|f\|_{L^2(C_1)}.
\end{equation*}
Because of \eqref{discrete_trivial}, this completes the proof.
\end{proof}

\subsection{A counterexample for  $ K^{\mathrm c}_t$}
For convenience we normalize  $ K^{\mathrm c}_t$ by multiplying by $(2\pi)^{-n/2}$.
Then we make an orthogonal change of variables in $x$ and in $u$ which diagonalizes the matrix $Q$. 
By also scaling the coordinates, we can then replace  $Q$ by the identity matrix. The new variables will be called $x'$ and $u'$, and we replace $f$ by $g$ defined by $f(u) = g(u')$, with equivalent $L^p$ norms. In the new coordinates, the relevant kernel is a tensor product
\begin{equation*}
  K_t^*(y') =  \prod_{i=1}^{n}  J_t(y'_i) 
\end{equation*}
where $J_t$ is the standard one-dimensional gaussian kernel
\begin{equation*}
   J_t(y'_i)  = \frac{1}{\sqrt{2\pi t}}\,\exp \Big(
{-\frac12\, t^{-1}
\left|  y_i'  \right|^2}\Big)
\end{equation*}
and the operator is given by $\mathcal H_t^*\, g = K_t^* * g$.

 We will choose $g$ as a tensor product
 $g(u') = \prod_{i=1}^{n} g_i(u'_i)$, so that
\begin{equation*}
  \mathcal  H_t^*\, g(x') =  \prod_{i=1}^{n} J_t  * g_j(x'_i),
\end{equation*}
with the convolutions taken in $\mathbb R$. In the sequel, we will consider the values of $\mathcal  H_t^* g$ only at points $x'$ in the cube $[0,1]^{n}$,  and each $g_i$ will have support in $[-1,2]$. 
This determines the compact sets $C_1$ and $C_2$, via the change of variables.
For
$i=2,\dots, n$, we simply choose $g_i = \chi_{[-1,2]}$, and observe that for these $i$ 
\begin{equation*}                       
  J_t  * g_i(x'_i) = 1+ \mathcal O(t), \qquad  x'_i \in [0,1].
\end{equation*}
It follows that 
\begin{equation}\label{i>1}
1 - \prod_{i=2}^{n} J_t  * g_i =\mathcal O(t) \qquad  \mathrm{in} \quad [0,1]^{n-1}.
 \end{equation}

The construction of $g_1$ requires more care. In the rest of this subsection,  we will   write simply
$x,\;u,\; y$ instead of $x'_1,\;u'_1,\;y'_1$ in one dimension. As already mentioned, we restrict  $t$ to the set $\{2^{-2\ell},\:\ell = 1,2, \dots \}$, and define operators
$$
A_\ell\, g_1 = J_{2^{-2\ell}} * g_1,\qquad \ell=1,2,\ldots.
$$ 
Notice that $J_{2^{-2\ell}}(y) = ({2\pi })^{-1/2}\,2^\ell\exp \left(
{-\frac12\, 
\left| 2^\ell y  \right|^2}\right)$.

We  will consider the variation in $\ell$    with $\ell$ ranging over the set
 $$
\mathcal I_N :=\{\ell \in \mathbb Z: 2N < \ell \le 3N \},
$$
 for large $N \in \N$.

To choose our function  $g_1$, we use the
 Rademacher functions,  supported in $[0,1]$ and given by
\begin{equation*}
  r_k = \sum_{j=1}^{2^{k-1}} \left( \chi_{((2j-2)2^{-k },\;(2j-1)2^{-k })} -  \chi_{((2j-1)2^{-k },\; 2j2^{-k })} \right), \qquad k=1,2,\dots.
\end{equation*}
We  identify  $[0,1)$ with the torus $\mathbb T = \R/\Z$, considered as a probability space with Lebesgue measure. Then the $r_k$ are independent Bernoulli random variables, taking the values $\pm 1$ with equal probability $1/2$.

In  $\R$ we define functions  coinciding  with the ${r}_k$ in $(0,1)$ and  supported in  $[-1,2]$, by setting for $k = 1,2,\dots$
 \begin{equation*}
   q_k(u) = r_k(u+1) + r_k(u) +r_k(u-1), \qquad  u \in \R.
 \end{equation*}

The function $g_1$ will  be
\begin{equation}\label{eq:def_gN}
g_N := \sum_{k \in \mathcal I_N} {q}_k.
\end{equation}
Notice that the notation $g_N$ is not in conflict with $g_i$, $i=1,\ldots,n$, since $N$ is large.
Then for  $1<p<\infty$ 
\begin{equation}\label{khinchine}
 \| g_N\|_{L^p(\R)} \simeq
   \| g_N\|_{L^p(\mathbb T)} \simeq \sqrt N,
\end{equation}
due to Khinchine's inequality.  We also observe the trivial estimate
 \begin{equation}\label{gNtriv}
   |g_N(x)| \le N, \qquad x \in \mathbb R.
 \end{equation}

The following is  the  main result of the  present subsection.
\begin{proposition}\label{controes2}
For some constant $c>0$, the Lebesgue measure of the set
  \begin{equation}
   \left\{x \in (0,1): \:\| A_\ell \,g_N(x) \|_{v(2),\mathcal I_N} > c \,\sqrt{N \,\mathrm{\log \log} N}\right\}
 \end{equation}
 tends to 1 as $N \to \infty$. Here the variation is taken in $\ell$.
\end{proposition}

\begin{proof}
Starting with $A_\ell$, we will change operator in several steps, until we arrive at an operator to which Qian's result in \cite{Qian} applies.

\vskip5pt

  {\textbf{Step 1}} will take us  from $A_\ell$ to the
   mean value operator $D_\ell,\: \ell= 1, 2,\dots $,   given by
\begin{equation*}
  D_\ell f(x) = 2^{\ell-1}\,\int_{x-2^{-\ell}}^{x+2^{-\ell}} f(y)\,dy.
\end{equation*}
This definition works equally well in $\R$ and in $\mathbb T$, and we will write  $D_\ell^\R$  and $D_\ell^{\mathbb T}$  to distinguish the two.

Jones and Wang \cite[Section 2]{Jones4} have introduced an equivalence relation $ \sim_p$ between sequences of operators acting on functions defined on the torus.
 We define a similar relation for operators $P_\ell$ and $Q_\ell$, $\ell \in \mathbb N_+$, acting   on functions in $\R$, with $p=2$.
  By   $(P_\ell) \sim_2 (Q_\ell)$ we mean
     that for any sequence $\big(\nu_\ell\big)_1^\infty $, with $|\nu_\ell| \le 1$,   and any $f \in L^2(\R)$
\begin{equation*}
  \|\sum_{\ell\ge 1}(P_\ell f - Q_\ell f)\, \nu_\ell \|_2 \lesssim \|f\|_2.
\end{equation*}
Here  the norm is that in $L^2(\R)$.
As shown in \cite[Theorem 2.5]{Jones4}, this implies that
\begin{equation*}
  \left\|\left(\sum_{\ell\ge 1}|P_\ell f - Q_\ell f|^2\right)^{1/2} \right\|_2 \lesssim \|f\|_2,
\end{equation*}
 which by \eqref{discrete_trivial} yields
 \begin{equation} \label{paragon}
  \left\| \left\|P_\ell f - Q_\ell f \right\|_{v(2),\N_+} \right\|_2 \lesssim \|f\|_2.
\end{equation}
\begin{lemma}\label{AD}
We have
  $$
  (A_\ell) \sim_2 (D_\ell^\R).
  $$
\end{lemma}
\begin{proof}
  We adapt the proof of \cite[Lemma 2.8]{Jones4} to $\R$.
  Observe first that  $A_\ell$ and $D_\ell^\R$ are convolution operators with kernels
  $J_{2^{-2\ell}} = (2\pi)^{-1/2}\, 2^\ell\, \exp\left(-\frac12\,\left(2^\ell\,y\right)^2\right)$ and
  $\chi_\ell(y) = 2^{\ell-1}\,\chi\left(2^\ell\,y\right)$, respectively,  $\chi$ denoting the characteristic function of the interval $(-1,1)$.
  We use Fourier transforms, defined by
$\mathcal F\,f (\xi)=\hat f(\xi)=
  \int f(x) e^{-2\pi i x\,\xi} \,dx$,
  and the Plancherel theorem, to deduce
         \begin{align*}
  \|\sum_{\ell\ge 1} (A_\ell f - D_\ell^\R f)\, \nu_\ell \|_2^2                                    
 & =   \int_{\mathbb R}
 \Big|
 \sum_{\ell\ge 1}
 \mathcal F(A_\ell f - D_\ell^\R f)(\xi)\, \nu_\ell\Big|^2 d\xi   \\
 & \le \int_{\mathbb R}|\hat f(\xi)|^2
 \Big(
 \sum_{\ell\ge 1}\big|
{\widehat {J_{2^{-2\ell}}}(\xi)-\widehat {\chi_\ell}(\xi)\big|}\Big)^2 d\xi.
\end{align*}
 It is thus enough to prove that
\begin{equation}\label{haer}
\sum_{\ell\ge 1}\big|
{\widehat {J_{2^{-2\ell}}}(\xi)-\widehat {\chi_\ell}(\xi)\big|}\lesssim 1.
\end{equation}
The Fourier transform of the kernel $J_{2^{-2\ell}}$ is
$$
\widehat  {J_{2^{-2\ell}}}(\xi)= \exp\left(-2\pi^2({2^{-\ell}\xi})^2\right),
$$
so that  $|1-\widehat {J_{2^{-2\ell}}}(\xi)| \lesssim |{2^{-\ell}}\xi|^2$ for $|2^{-2\ell}\xi| \le 1$ and
$|\widehat {J_{2^{-\ell}}}(\xi)| \lesssim |2^{-\ell}\xi|^{-2}$ for $|2^{-\ell}\xi| \ge 1$.
The Fourier transform
$$
\widehat {\chi_\ell}(\xi) = \frac{\sin (2\pi2^{-\ell} \xi)}{2\pi 2^{-\ell}\xi}
$$
similarly satisfies   $|1-\widehat {\chi_{{\ell}}}(\xi)| \lesssim |{2^{-\ell}}\xi|^2$ for
$|2^{-\ell}\xi| \le 1$ but only $|\widehat {\chi_{{\ell}}}(\xi)| \lesssim |2^{-\ell}\xi|^{-1}$ for $|2^{-\ell}\xi| \ge 1$. From this  \eqref{haer} and the lemma follow.
\end{proof}

To   $A_{\ell} (g_N) (x)$ and $D_{\ell}^\R (g_N) (x)$ we  now apply the triangle inequality
for the variation seminorm and write for  $x \in \R$
 \begin{align} \label{ADR}
  \left\| A_{\ell}  (g_N) (x)  \right\|_{v(2),\mathcal I_N} \ge
    \left\| D_{\ell}^{\R} ( g_N)  (x) \right\|_{v(2),\mathcal I_N} - h_1 (x),
\end{align}
where
\begin{equation*}
  h_1  (x) =  \left\|A_{\ell}  ( g_N) (x) -  D_{\ell}^{\R} ( g_N) (x) \right\|_{v(2),\mathcal I_N}.
\end{equation*}
From Lemma  \ref{AD},  \eqref{paragon} and \eqref{khinchine}, we conclude that
$\|h_1\|_{L^2(\R)} \lesssim \sqrt N$.

\vskip17pt

  {\textbf{Step 2}}  consists in passing from $\R$ to $\mathbb T$. Observe that for any point $x \in (0,1)$, also considered as a point in   $\mathbb T$, one has
$D_{\ell}^{\R} (q_k)(x) = D_\ell^{\mathbb T} (r_k)(x)$.
Summing in $k$, we get for these $x$
$$
D_{\ell}^{\R} ( g_N) (x) = D_{\ell}^{\mathbb T} (T_N)(x),
$$
 where the function
 $$
 T_N = \sum_{k \in \mathcal I_N} r_k
 $$
  is defined in  $\mathbb T$. 
From \eqref{khinchine} we conclude
\begin{equation}\label{khinch}
  \| T_N   \|_{L^2(\mathbb T)} \simeq \sqrt N.
\end{equation}

\vskip17pt

For {\textbf{Step 3}}, we
let $E_{\ell}$ for $\ell= 1,2,\dots$ denote the conditional expectation operator
 which replaces an integrable function defined in  $\mathbb T$ by its mean value in each dyadic subinterval $((j-1)2^{-\ell },\; j2^{-\ell }) \subset [0,1]$. Then
\begin{equation}\label{Tn}
  E_{\ell} (T_N)=  \sum_{k \in \mathcal I_N, \: k \le \ell} r_k
\end{equation}
for $ \ell\in \mathcal I_N$, so that the $E_{\ell} (T_N)$
 form a finite martingale for the sequence of $\sigma$-algebras generated by the dyadic intervals of length $2^{-\ell}$, 
$\ell \in \mathcal I_N$, contained in $[0,1]$.
Step 3 aims at
 replacing $D_{\ell}^{\mathbb T}$ by  $E_{\ell}$.

Jones and Rosenblatt 
prove in \cite[Remark 4]{Jones_Rosenblatt} that for functions $f \in L^2(\mathbb T)$
  \begin{equation*}
  \left\|\left(\sum_{\ell=1}^\infty |D_{\ell}^{\mathbb T} f - E_{\ell} f|^2\right)^{1/2} \right\|_2 \lesssim \|f\|_2.
\end{equation*}
Here we refer to  the norm in $L^2(\mathbb T)$. These authors use a one-sided, right mean value operator (see \cite[p.528]{Jones_Rosenblatt}),
but since our $D_{\ell}^{\mathbb T}$ is the mean of the right and left operators, this is no problem.

Letting here $f = T_N$, we can apply \eqref{discrete_trivial} and then \eqref{khinch}, to get
 \begin{equation*}
   \left\| \left\|D_{\ell}^{\mathbb T} ( T_N) - E_{\ell} ( T_N)\right\|_{v(2),\mathcal I_N}\right\|_2
\le  \left\|\left(\sum_{\ell \in \mathcal I_N} |D_{\ell}^{\mathbb T} ( T_N) - E_{\ell} ( T_N)|^2\right)^{1/2} \right\|_2 \lesssim \|T_N\|_2  \simeq \sqrt N.
\end{equation*}

The triangle inequality for the seminorm $\|\cdot\|_{v(2),\mathcal I_N}$ now implies  that for  $x \in \mathbb T$
 \begin{equation} \label{DT}
  \left\| D_{\ell}^{\mathbb T}  ( T_N)(x) \right\|_{v(2),\mathcal I_N} \ge
 \left\| E_{\ell} ( T_N) (x) \right\|_{v(2),\mathcal I_N}
  - h_2(x),
\end{equation}
with
\begin{equation}\label{h2}
   \|h_2\|_{L^2(\mathbb T)} \lesssim \sqrt N.
              \end{equation}

\vskip17pt

In  {\textbf{Step 4}} we observe that the $r_k$ with $k \in \mathcal I_N$ are $N$ independent random variables fulfilling the hypotheses of
Qian \cite[Theorem 2.2]{Qian}. Because  of \eqref{Tn},  this result
 tells us that the probability, i.e., the Lebesgue measure, of the set
\begin{equation*}
 \left\{x \in \mathbb T: \:\|E_{\ell} (T_N)(x) \|_{v(2),\mathcal J_N} > c \,\sqrt{N \,\mathrm{\log \log} N}\right\}
\end{equation*}
tends to 1 as $N \to \infty$, for some $c>0$.

\vskip17pt

Finally, {\textbf{Step 5}} combines the result of Step 4 with the preceding steps, in reverse order.  Considering Step 3, we see that \eqref{h2}  and Chebyshev's inequality imply that
$|h_2(x)|$ is much smaller than $\sqrt{N\, \mathrm{log \log} N}$ except on a subset of $\mathbb T$ whose measure tends to 0  as $N \to \infty$.
Then \eqref{DT} implies that the measure of the set
\begin{equation*}
 \left\{x \in \mathbb T: \:\|D_{\ell}^{\mathbb T} (T_N)(x) \|_{v(2),\mathcal I_N} > c \,\sqrt{N\, \mathrm{\log \log} N}\right\}
\end{equation*}
  tends to 1 as $N \to \infty$, for some $c>0$.

Now  Step 2 says that this set is the same as
\begin{equation*}
 \left\{x \in [0,1): \:\|D_{\ell}^{\R} (g_N)(x) \|_{v(2),\mathcal I_N} > c \,\sqrt{N \,\mathrm{\log \log} N}\right\}.
\end{equation*}

As for Step 1, we can apply \eqref{ADR} with $x \in [0,1)$ and argue as we did for Step 3. Thus the set
\begin{equation} \label{A>}
 \left\{x \in [0,1): \:\|A_{\ell} (g_N)(x) \|_{v(2),\mathcal I_N} > c \,\sqrt{N \mathrm{\log \log} N}\right\}.
\end{equation}
also has  measure tending to 1, for some $c>0$.

The proof of Proposition \ref{controes2} is complete.
\end{proof}

 \subsection{End of proof of Theorem \ref{thm_counterex}}
 The function  $g=\otimes_{i=1}^n g_i$, where $g_1=g_N$, 
  satisfies 
  for $x' \in [0,1]^{n}$
 \begin{align*}
   \mathcal H_t^* \,g(x')            
 &  = J_t  * g_N(x'_1)\, \prod_{i=2}^{n} J_t  * g_i(x'_1)\\
&   = J_t  * g_N(x'_1) - J_t  * g_1(x'_1)\, \left(1-\prod_{i=2}^{n} J_t  * g_i(x'_i)\right)
 \end{align*}
 To estimate the second term here, we apply \eqref{gNtriv} to its first factor and \eqref{i>1} to the second factor. The second term is thus no larger than constant times $Nt$, which is dominated by
 $2^{-N}$ if $t = 2^{-2\ell}$ with $\ell\in \mathcal I_N$. From \eqref{discrete_trivial} we see that the
 seminorm
$\|\cdot\|_{v(2),\mathcal I_N}$ 
 of the second term is no larger than $N^{1/2}2^{-N}$, since the cardinality of $\mathcal I_N$ is $N$. 
Then we apply Proposition \ref{controes2} to 
 the first term, and conclude that the statement of this proposition also holds for $\mathcal H_t^* \,g(x')$.
 
The same is then true for $\mathcal H_t^c f(x)$,
where $f(u)=g(u')$ as before, and $x$ is now in some compact set.
Finally, Proposition \ref{lemmaIIt}
allows us to replace 
$\mathcal H_t^c $
by the operator defined by the kernel 
$\widetilde K_t (x,u)$,
and Theorem \ref{thm_counterex} follows.

\vskip15pt

\section{Appendix}\label{Appendix}

We consider    
 the integral in \eqref{formula2_Bochner}.

\begin{lemma}
If  $f \in L^1(du)$  and  $x$ is a point such that  $x\notin \mathrm{supp}\, f$, then
\begin{equation} \label{integral}                      
\int \widetilde M^{\rm{loc}}(x,u)\,f(u)\,du
\end{equation} 
 is a Bochner integral with respect to the Banach space $F_\varrho$.
\end{lemma}

\begin{proof}
The function   $f$ and the point $x\notin \mathrm{supp}\, f$ will be fixed.
   Proposition \ref{lemma-Calderon-no}($a$)  shows that     
   the integrand        in \eqref{integral}                       
   is in $F_\varrho$ for each $u$.
   
   We will construct 
   a family  $S_r(u)$, with $0<r<r_0$ for some $r_0>0$, of $F_\varrho$-valued simple functions such that for a.a. $u$
 \begin{equation}	\label{measurability}
 \left\| \widetilde M^{\rm{loc}}(x,u)\,f(u) - S_r(u) \right\|_{F_\varrho} \to 0 \qquad  \mathrm{as}\quad r \to 0
\end{equation}
(the integrand is then called strongly measurable),  and also
\begin{equation}\label{integrability}
 \int \left\| \widetilde M^{\rm{loc}}(x,u)\,f(u) - S_r(u) \right\|_{F_\varrho}\,du \to 0 \qquad  \mathrm{as}\quad  r \to 0.
\end{equation}    
 Then \eqref{integral} fulfills the definition of an $F_\varrho$-Bochner integral, see \cite[Appendix~C]{Engel}.


To find these simple functions, we start by covering $\mathrm{supp} \,f$ with
 small cubes $Q_\nu\,,$ pairwise disjoint and all of side $r$. Here $r>0$ is so small that the distance from $x$   to the union of these cubes is larger than some $\delta_0 > 0$ independent of $r$.  The center of  $Q_\nu$ is denoted $d_\nu$.

In this proof, we allow all implicit constants to depend on $x,\;\delta_0$ and $f$.

The integrand $ \widetilde M^{\rm{loc}}(x,u)\,f(u)$ will be approximated by       
\begin{equation*}
S_r(u) = \sum_\nu \widetilde M^{\rm{loc}}(x,d_\nu)\, \frac 1 {|Q_\nu|}\int_{Q_\nu} f(y)\,dy\: \chi_{Q_\nu}(u).
\end{equation*}
We then have
\begin{align*}
\widetilde M^{\rm{loc}}(x,u)&\,f(u) - S_r(u)\\ &=
\sum_\nu  \frac 1 {|Q_\nu|}\int_{Q_\nu}
\left( \widetilde M^{\rm{loc}}(x,u) \,f(u)- \widetilde M^{\rm{loc}}(x,d_\nu)\, f(y) \right)\,dy\: \chi_{Q_\nu}(u)\\
&=
 \sum_\nu \widetilde M^{\rm{loc}}(x,d_\nu)\, \frac 1 {|Q_\nu|}\int_{Q_\nu} (f(u)-
 f(y))\,dy
 \,\chi_{Q_\nu}(u) \\ &+
 \sum_\nu \left(\widetilde M^{\rm{loc}}(x,u)
 -\widetilde M^{\rm{loc}}(x,d_\nu)\right)\,  \frac 1 {|Q_\nu|}\int_{Q_\nu} \,dy\,f(u)\: \chi_{Q_\nu}(u) 
 =: \mathcal I +  \mathcal  J .
  \end{align*}
  
  Starting with $ \mathcal J$, we claim that
                       $\|\widetilde M^{\rm{loc}}(x,u) -\widetilde M^{\rm{loc}}(x,d_\nu) \|_{F_\varrho} = \mathcal O(r)$ as $r \to 0$. 
 As in the proof of   Proposition \ref{lemma-Calderon-no}($a$)  we have
 \begin{align*}
   \|\widetilde M^{\rm{loc}}(x,u) -\widetilde M^{\rm{loc}}(x,d_\nu) \|_{F_\varrho}&=
    e^{-R(x)}\, \|\eta(x,u)\,K_t(x,u) - \eta(x,d_\nu)\,K_t(x,d_\nu)\|_{v(\varrho),(0,1]} \\& \le 
    e^{-R(x)}\, |\eta(x,u) -\eta(x,d_\nu)|\,\|K_t(x,u)\|_{v(\varrho),(0,1]} \,\\  &+ \,
     e^{-R(x)}\, \eta(x,d_\nu)\, \|K_t(x,u) - K_t(x,d_\nu)\|_{v(\varrho),(0,1]}\,. 
 \end{align*}
 In the sum here, the first term is no larger than constant times $e^{-R(x)}\, r\,\int_{0}^{1}| \dot K_t(x,u)|\,dt \lesssim r\,|x-u|^{-n}       = \mathcal O(r)$,  again because of  
 Proposition \ref{lemma-Calderon-no}($a$) and its proof. The second term does not exceed
 \begin{equation*}
  e^{-R(x)}\, \int_{0}^{1}| \partial_t(K_t(x,u) - K_t(x,d_\nu))|\,dt\,. 
 \end{equation*}
 If  $w$ denotes the  vector $u - d_\nu$ and   $\nabla_2$
 the gradient in the second variable, we can write
 \begin{align*}
    \partial_t(K_t(x,u) - K_t(x,d_\nu)) &
    = \partial_t \, \int_{0}^{1} \left\langle    \nabla_2
    \,K_t(x,d_\nu + sw), w \right\rangle \,ds\\
   & = \int_{0}^{1} \left\langle
   \nabla_2
    \,\dot K_t(x,d_\nu + sw), w \right\rangle \,ds.
 \end{align*}
 The estimate \eqref{dtdu} shows that
 \begin{align*}
  e^{-R(x)}\,    \int_{0}^{1} |
     \nabla_2
     \,\dot K_t(x,d_\nu + sw)|\,ds \lesssim 1.
 \end{align*}
  Since  $|w| \lesssim r$, the claim follows, and we see that  
  \begin{equation*}
\left\|        \mathcal J    \right\|_{F_\varrho}    \lesssim C r \,|f(u)|  \to 0  \qquad \mathrm{a.e.}  
  \end{equation*}


 As for the term $ \mathcal I $,  Proposition \ref{lemma-Calderon-no}($a$) implies that
\begin{align*}
 \| \mathcal I \|_{F_\varrho} \lesssim \sum_\nu  \frac 1 {|Q_\nu|}\int_{Q_\nu} \,|f(u)- f(y)|\,dy   \: \chi_{Q_\nu}(u),
 \end{align*}
 whence 
$  \| \mathcal I  \|_{F_\varrho}\to 0$ as $r\to 0$  at any Lebesgue point.

Thus the measurability condition \eqref{measurability} is satisfied.

 It remains to verify \eqref{integrability}.
 Reasoning as above, one has  since  $|w| \lesssim r$
 \begin{equation*}
  \int \left\|        \mathcal J     \right\|_{F_\varrho} \,du   = \mathcal O\left(\| f\|_{L^1}\,r \right).
  \end{equation*}
   For the term $ \mathcal I $, we write
  \begin{align*}
   \int   \left\|      \mathcal I       \right\|_{F_\varrho}\,du  & \lesssim
   \sum_\nu \int_{Q_\nu} \frac 1 {|Q_\nu|}\int_{Q_\nu} \,|f(u)- f(y)|\,dy \,du \\
& \le   \sum_\nu \int_{Q_\nu}  \frac 1 {|Q_\nu|}\int_{|v|< r\sqrt n} \,|f(u)- f(u+v)|\,dv \,du  \\
& =   r^{-n}  \,\int_{|v|< r\sqrt n} \,dv \,  \sum_\nu  \, \int_{Q_\nu} |f(u)- f(u+v)|\,du,
\end{align*}
since $|Q_\nu| = r^n$. Denoting by $\tau_v$ the translation operator, we see that the last expression in this formula equals
\begin{align*}
   r^{-n}  \,  \int_{|v|< r\sqrt n} \,dw \, \|f- \tau_v \,f \|_{L^1} \lesssim \sup_{|v|< r\sqrt n} \|f- \tau_v \,f \|_{L^1}.
\end{align*}
Since translation is continuous in the $L^1$ norm, this tends to 0 with $r$.

Altogether, we have proved \eqref{integrability},
  and  so the integral \eqref{integral} is ${F_\varrho}$-Bochner.
\end{proof}

\subsection*{Conflict of interest }
The authors
declare that they do not have any conflict of interest for this work.

\end{document}